\documentclass[12pt]{article}

\usepackage{geometry}
\usepackage[english]{babel}

\usepackage{amsmath,amssymb}
\usepackage{graphicx}
\usepackage{authblk}
\usepackage{multicol}
\usepackage[colorlinks=true,linkcolor=blue,citecolor=blue,urlcolor=blue]{hyperref}
\usepackage[numbers]{natbib}

\usepackage{amsmath,bm}
\usepackage{amssymb,amsfonts}
\usepackage{amsthm}
\usepackage{mathrsfs}
\usepackage{graphicx}
\usepackage{textcomp}
\usepackage{comment}
\usepackage{algorithm}
\usepackage{array}
\usepackage{subcaption}
\usepackage{pgffor}

\usepackage{siunitx}
\usepackage{subcaption}
\usepackage{mathtools}
\usepackage{algorithm}
\usepackage{algorithmic}
\usepackage{booktabs}
\usepackage{todonotes}
\usepackage{enumerate}

\DeclarePairedDelimiter{\flexpar}{(}{)}

\DeclarePairedDelimiter{\norm}{\lVert}{\rVert}
\newcommand{\Hvec}{\bm{H}}
\newcommand{\Hxt}{\Hvec (\bm{x},t)}
\newcommand{\Mvec}{\bm{M}}
\newcommand{\Madvec}{\bm{M}_{\mathrm{ad}}}
\newcommand{\Mxt}{\bm{M}(\bm{x},t)}
\newcommand{\Mxtad}{\bm{M}_{\mathrm{ad}}(\bm{x},t)}
\newcommand{\Hxtn}{\frac{\Hxt}{\lVert \Hxt\rVert}}
\newcommand{\Hhsat}{\frac{\lVert \Hxt \rVert}{\hsat}}
\newcommand{\xvec}{\bm{x}}
\newcommand{\hsat}{H_{\mathrm{sat}}}
\newcommand{\msat}{M_{\mathrm{sat}}}

\newcommand{\yvec}{\bm{y}}
\newcommand{\rvec}{\bm{r}}
\newcommand{\rvect}{\bm{r}(t)}
\newcommand{\rveck}{\bm{r}_k}
\newcommand{\vvec}{\bm{v}}
\newcommand{\vvect}{\bm{v}(t)}
\newcommand{\vveck}{\bm{v}_k}
\newcommand{\svec}{\bm{s}}
\newcommand{\sadvec}{\bm{s}_{\mathrm{ad}}}
\newcommand{\svect}{\bm{s}(t)}
\newcommand{\sveck}{\bm{s}_k}
\newcommand{\sadveck}{\bm{s}_{\mathrm{ad},k}}
\newcommand{\sadvecn}{\bm{s}_{\mathrm{ad},n}}

\newcommand{\avec}{\bm{a}}

\newcommand{\analogft}{\bm{a}_\mathrm{AF}(t)}
\newcommand{\Langevin}{\mathcal{L}}
\newcommand{\diag}{\mathrm{diag}}
\newcommand{\Deltat}{\Delta t}
\newcommand{\hpf}{f_{\mathrm{HP}}}
\newcommand{\regularizer}{\mathcal{R}}
\newcommand{\interpolator}{\mathcal{I}}
\newcommand{\e}{\mathrm{e}}
\newcommand{\R}{\mathbb{R}}
\newcommand{\N}{\mathbb{N}}
\DeclareMathOperator{\vect}{vec}

\usepackage{pifont}
\newcommand{\cmark}{\checkmark}
\newcommand{\xmark}{\times}

\newtheorem{theorem}{Theorem}[section]

\newtheorem{lemma}[theorem]{Lemma}

\title{Debye Relaxation in Model-Based Multi-Dimensional Magnetic Particle Imaging}

\author[1,2]{Vladyslav Gapyak \thanks{Emails: vladyslav.gapyak@proton.me, thomas.maerz@h-da.de, andreas.weinmann@thws.de}}
\author[1,2]{Thomas März}
\author[3]{Andreas Weinmann}
\affil[1]{\footnotesize Hochschule Darmstadt, Schöfferstraße 3, 64295 Darmstadt, Germany}
\affil[2]{Data Science Institute, European University of Technology, European Union}
\affil[3]{Algorithms for Computer Vision, Imaging and Data Analysis Lab, Technische Hochschule Würzburg-Schweinfurt, Ignaz-Schön-Straße 11, 97421 Schweinfurt, Germany}

\date{}

\begin{document}
	\maketitle
	
\begin{abstract}
Model-based reconstruction approaches for the medical imaging modality Magnetic Particle Imaging (MPI) are typically based on the Langevin model, which assumes instantaneous alignment of the particles magnetic momenta with the applied field. Regarding the application to real data, Langevin model-based reconstruction methods require model transfer functions (MTF) obtained from calibrations to preprocess the data. There are also model-based reconstruction approaches that include relaxation effects and other particle-level dynamics. However, they are limited either to 1D or 1D-like scanning scenarios when considering real data, or are limited to simulated data in the case of multi-dimensional field-free point  (FFP) MPI. Thus, fully model-based reconstructions from multi-dimensional FFP scanning data that incorporate relaxation effects without using an MTF have not yet been demonstrated. In this work, we incorporate relaxation effects by considering a multi-dimensional Debye model and provide reconstruction formulae. In particular, we show that the Debye model-based signal is the response of a linear time-invariant system with exponential memory applied to a Langevin model-based signal. We provide a reconstruction algorithm for the introduced multi-dimensional Debye model. To this end, we devise a relaxation adaption step. For the resulting relaxation-adapted Debye signal, we show that it can be expressed by the well-studied MPI core operator derived from the Langevin theory. This results in a three-stage algorithm with low additional cost over the Langevin model, 
as the relaxation adaption scales linearly in the input data. We provide numerical results for the proposed algorithmic approach. In particular, we obtain fully model-based reconstructions from real 2D MPI data without involving any specific MTF analogous to the Langevin model case.
\end{abstract}
	\vspace{1em}
	

\section{Introduction}

Magnetic Particle Imaging (MPI) is a tracer-based imaging modality invented by Gleich and Weizenecker 
and first published in~\cite{gleich2005original}. The tracers used in MPI are ferrofluids containing 
superparamagnetic iron-oxide nanoparticles (SPIONs), which have the property to respond to time-varying 
magnetic fields. Their magnetization response induces a voltage in the receive coils of the MPI scanner; 
upon digitalization, the induced voltage constitutes the signal, or measurement, obtained during an MPI scan. 
The objective of this imaging modality is to localize the particles, i.e., to reconstruct their spatial distribution 
from the obtained signal. In order to localize the nanoparticles, the magnetic fields applied usually follow a 
precise structure: they are the superposition of a \textit{selection field} and of a 
\textit{drive field}~\cite{BuzugKnopp2012}. The selection field is a static magnetic field with constant gradient, 
which magnetically saturates the concentration of nanoparticles, and vanishes in a field-free region. 
A dynamic field, the drive field, is superimposed onto the static field and steers the field-free region within 
the field-of-view (FOV) of the scanner. Because the SPIONs are magnetically saturated everywhere but close to the 
field-free region and their response is highly non-linear close to this region, the field-free region acts as a 
sensitive spot and allows the localization of the particles constituting the target concentration. 
This principle has proven itself successful and in 2009, Weizenecker et al.~\cite{Weizenecker_etal2009} 
showed a video reconstruction (3D+time) of the particle distribution within the beating heart of a mouse, 
setting a milestone in MPI imaging and showing its applicability to real-time in vivo imaging. 
These achievements are possible thanks to MPI's high sensitivity and fast acquisition time~\cite{BuzugKnopp2012}. 
Additionally, MPI has the advantage that it does not employ ionizing radiation, in contrast to other widespread 
medical imaging modalities~\cite{irfan2022}.
Furthermore, researchers have proposed a range of different medical applications of MPI; among these we mention 
the detection of as few as 250 cancer cells~\cite{Song2018}, the tracing of stem 
cells~\cite{connell2015advancedcellTherapies,GoodwillConolly2011,JUNG2018139,Lemaster2018,Tomitaka2015Lactoferrin} 
as well as blood-flow~\cite{Franke2020BloodFlow} and cardiovascular imaging~\cite{Bakenecker2018MPIvascular,Tong2021,Vaalma2017}. 
Further MPI applications can be found in~\cite{billingsMPIapplications,Yang2022}.

Concerning the reconstruction algorithms, they can be divided into two main classes: 
measurement-based or model-based approaches. Measurement-based approaches usually involve the calibration of a 
system-matrix: the FOV is subdivided into a grid and a known concentration of particles -- the $\delta$-concentration -- 
is iteratively put, scanned, and moved across all cells of the chosen grid; the scans of the $\delta$-concentration 
are stored as columns of a system matrix $S$, which consequently contains the responses of the scanning system 
to $\delta$-impulses; given the scan data $s$ collected by scanning a target concentration, the regularized inversion 
of the linear system $Sx=s$ yields the discrete approximation $x$ of the target concentration, represented as 
a piece-wise constant approximation which is constant on each cell of the chosen grid of the FOV. 
The advantage of measurement-based methods is that the physical phenomena and inherent complexities of the 
imaging setup are implicitly taken into account by the calibration. The disadvantage is that such calibrations 
are time-consuming. For this reason, effort has been put into developments of methods that can reduce the 
calibration time. Such methods include the super-resolution of system matrices acquired on low-resolution 
grids~\cite{gungor2021superresolution,Schrank2022superresolution,Yin2023dipSM}, extrapolation of system matrices 
outside the FOV~\cite{schefflerboundary2022,scheffler2023extrapolation}, and techniques involving compressed 
sensing~\cite{KnoppWeber2013,Weber2015_symm_compressed}.
Complementary to calibration acceleration methods, model-based approaches aim at performing reconstructions 
without the need of a system matrix; this is achieved by modeling the forward imaging operator, starting from 
the mathematical equations governing the physical phenomena underlying MPI. In addition, the modeling of the 
scanning geometry becomes important: the filed-free region that acts as a sensitive spot can either be a point 
-- the field-free point (FFP) geometry -- or a whole line -- the field-free line (FFL) acquisition 
geometry~\cite{ErbeFFLbook}. We mention that due to the constraints provided by Maxwell's equations, higher 
dimension geometries such as a hypothetical field-free plane are not possible~\cite{gapyak2023ffl3d}. 
In this work we will consider the FFP scanning geometry. Concerning the modeling of magnetic phenomena in MPI, 
the simplest models use the Langevin theory of paramagnetism~\cite{jiles1998introduction}. 
The Langevin model has been widely studied in the context of model-based MPI reconstructions and it is based 
on adiabatic assumptions and instantaneous alignment of the particles to the magnetic field. Real nanoparticles 
however show more complex behavior and in particular, they exhibit delays in the realignment of the magnetic 
momenta, a phenomenon called \textit{relaxation}. In the following section we offer an overview of model-based 
reconstruction approaches (with or without relaxation) to contextualize our contributions within the MPI literature. 

\subsection{Related work}

\paragraph{Time-domain model-based approaches (X-space family).}
The earliest example of model-based reconstruction in time-domain is the X-space formulation by Goodwill 
and Conolly~\cite{GoodwillConolly2010}, and it uses the Langevin model. The original 1D formulation (2010) 
was extended in 2011 to multidimensional settings~\cite{GoodwillConolly2011}, which however rely on Cartesian 
trajectories in the 2D/3D case, i.e., trajectories that are flat and approximate the 1D case. In 2019 extensions 
of the X-space framework to non-Cartesian trajectories was proposed by Ozaslan et al.~\cite{ozaslan2019regridding}, 
using a regridding strategy. The experiments were performed in simulated scenarios only. Reconstruction formulae 
for the $n$-dimensional case were introduced by März and Weinmann in 2016~\cite{marz2016model}. In~\cite{marz2016model} 
the concentration-to-signal model is formulated in terms of the MPI core operator, which mediates the concentration, 
the positions, velocities and the measured signal in time domain; additionally, explicit reconstruction formulae are 
derived and it is realized as a two-stage algorithm (core stage and deconvolution stage). 
These reconstruction formulae were tested in a variety of simulated scenarios, such as in 
multi-patch MPI~\cite{gapyak2023multipatch} and extended to 3D FFL scenarios~\cite{gapyak2023ffl3d}. 
The recent MoBiT-2S algorithm~\cite{gapyak2025trajectoryindep} builds on the two-stage algorithm and 
demonstrates FFP model-based reconstruction on real 2D MPI data with both data obtained with Lissajous 
scans (``MPIData: EquilibriumModelwithAnisotryopy"~\cite{knopp2024equilibriumdata} Bruker dataset) 
as well as on non-Lissajous type data~\cite{leili2025transverseMNP}. In the experiments with MoBiT-2S 
on Bruker data, a computed transfer function was used to divide the signal in Fourier-domain, similarly 
to what is normally done in other model-based publications with Lissajous data 
(e.g., with the Chebyshev polynomial~\cite{droigk2022multidimcheb,droigk2025efficientcheb}). 
In this sense, the reconstruction pipelines are hybrid and perform model-based reconstructions on data 
that have been cleaned using information contained in a calibrated system matrix.
Recently, Sanders et al.~\cite{sanders2025physicsbased} published a reformulation of the forward model 
as the product of sparse linear operators. Reconstruction is performed by regularized inversion of the resulting 
linear operator. High quality results are shown for 2D FFL scans obtained with the Momentum scanner (Magnetic Insight), 
which reconstruct a 2D projection image (orthogonally to the FFL) with a FFL moving in a Cartesian-like pattern, 
thus leveraging 1D-like scanning trajectories to employ the X-space-based algorithms. In~\cite{sanders2025physicsbased} 
it was also proposed to use the method for 2D Lissajous scans (in a simulated scenario).

\paragraph{Frequency-domain model-based approaches (Chebyshev family) and system-function family.}
We discuss model-based reconstruction algorithms that consider the Langevin model and the signal in Fourier-domain. 
The first is due to~\cite{Rahemeretal2009}, where it has been shown that the 1D MPI signal can be represented in 
terms of Chebyshev polynomials of the second kind; this representation has been leveraged to devise reconstruction 
schemes in the Chebyshev basis. A rigorous analysis of the 1D Langevin model has been given by Erb 
et al.~\cite{erb2018mathematical}, characterizing also the decay of the singular values of the operator and 
its ill-posedness. For the multi-dimensional FFP case, Maa{\ss} and Mertins~\cite{maass2019fourier} showed 
that the Fourier coefficients of the MPI signal can be represented as a series in terms of Chebyshev polynomials 
of the second kind; however, these results are tied to the the specific case of Lissajous scanning trajectories 
(sinusoidal along each coordinate). Subsequently, Droigk et al.~\cite{droigk2022multidimcheb} used the direct 
Chebyshev reconstruction method (DCR) stemming from the formulae in~\cite{maass2019fourier} to produce reconstruction 
results on 2D FFP data acquired with Bruker's scanner, which indeed employs Lissajous scanning curves. 
However, these results are not only tied to the Lissajous scanning trajectory, they also employ a transfer 
function to preprocess the data: following Knopp et al.~\cite{KnoppBiederer_etal2010}, the Fourier-spectrum of 
the signal is divided by a transfer function, estimated by fitting simulated calibration data with the Langevin 
model and a measured system matrix in a least-squares fashion. In this sense, the results have been shown in a 
hybrid scenario: although no calibrated system matrix is used in the reconstruction algorithms, it is employed in 
the preprocessing of the data to render the input data fit for the method. Improvements on the model side where 
achieved with the introduction of the equilibrium model with anisotropy (EQANIS)~\cite{maass2024equilibriumanysotropy}, 
which extends the Langevin model to include anisotropy of the nanoparticles and has been used to simulate system 
matrices for system inversion~\cite{maass2024equilibriumanysotropy} as well as with the Chebyshev representation 
in Lissajous trajectories~\cite{droigk2025efficientcheb}. However, these results as well were obtained by 
preprocessing the data with a computed transfer function obtained from a measured system-matrix. 

\paragraph{X-space Debye relaxation models.} 
Relaxation effects in X-space frameworks in MPI were first incorporated in essentially one-dimensional settings. 
In particular, in 2012 Croft et al.~\cite{Croft2012relaxation} amended the 1D X-space theory to account for 
non-adiabatic SPIO magnetization effects. In particular, they adopted the first-order Debye model following earlier 
literature on ferrofluids such as~\cite{shliomis1974}. The first-order Debye process models relaxation by incorporating 
a \textit{temporal} convolution between the adiabatic Langevin magnetization and an exponentially decaying Debye kernel.
The subsequent analysis and experiments have been carried out in essentially 1D scenarios: on the Berkley 
relaxometer (producing a virtual FFP along a line) and on the Berkley projection X-space scanner, which uses an 
FFL moving along a single scan direction. They show that the predicted asymmetric and direction-dependent 
blurring of the 1D point-spread function (PSF) agrees with the measured data, and deconvolution is discussed 
as a possible post-processing step. In 2015, Bente et al.~\cite{bente2015rotFFL} extended the idea and incorporated 
the deconvolution with the Debye kernel to projection data acquired with an electronically rotated FFL scanner. 
In particular, in~\cite{bente2015rotFFL} the signal is corrected using the 1D Debye kernel along each of the 
multiple FFL scanning directions before reconstruction. However, although these results underpin the benefits 
of incorporating relaxation effects, the proposed methods were still confined to 1D-like FFL or 1D FFP setups.
Subsequent work by Utkur et al.~\cite{utkur2017viscosity} and Muslu et al.~\cite{utkur2018multicolor} on 
relaxation-based viscosity mapping and multicolor MPI also operate in 1D-like scenarios 
(magnetic particle spectroscopy or stacked 1D scans).

\paragraph{Particle-level Debye-type and hysteresis models.} 
The behavior of real nanoparticles is complex and a range of forward models have been 
formulated~\cite{kluth2018mathematical}. Such models range from the Langevin formulation (equilibrium model) 
to formulations that incorporate the rotation of the whole particle (Brownian rotation~\cite{coffey1992brown}), 
or the internal rotation (N\'eel relaxation~\cite{neel1953relaxation}), as well as models that employ the more 
general Fokker-Plank equation~\cite{risken1989fokker}, which are however computationally intensive. 
In 2023 Solibida et al.~\cite{solibida2023debye} introduce the refined Debye model (RDM) in which the relaxation 
time is dependent on the magnetic field applied (and hence on time). This field-dependent relaxation is fitted 
via Brown and N\'eel stochastic simulations and compared with the classical Debye model, showing that the 
RDM is able to describe the behavior of the harmonics for different viscosities, wheras the standard Debye 
model fails. However, the benefits of RDM have been shown on a 1D multifrequency MPI setup and no image 
reconstruction on multi-dimensional scanners was performed. In 2023 Li et al. contributed in two additional 
ways to Debye-like models in MPI; first, they introduce a multi-dimensional Debye model~\cite{li2023multidebye}, 
where the relaxation kernel is the superposition of several first-order Debye terms, weighted by the relative 
amplitudes of the magnetic fields's components. Comparison between simulated and measured magnetization curves 
show the improvement of the multi-dimensional Debye model over the standard Debye model; although it incorporates 
the influence of different components of the magnetic fields, the signal is collected only along the $x$ axis, 
and the multi-dimensional drive-field amplitudes along the $x$ and $y$-axis lay in a 10,000:20 or 10:1 ratio, 
which results in flattened trajectories that approximate 1D scans. Moreover, no images of real phantoms are 
presented. The second contribution of Li et al. involves the introduction of the modified Jiles-Atherton 
(MJA) model~\cite{li2023jiles-arthenton}, whose magnetization curves (and its derivative) show strong 
agreement with the measured magnetization; the reconstructions with the specialized MJA X-space algorithm 
have been performed in a 1D scenario and the applicability to multi-dimensional MPI with 2D (or 3D) scanning 
sequences remains open.

\subsection{Contribution}
In summary, existing model-based approaches can be grouped into three classes:
\begin{enumerate}[i)]
	\item approaches that include relaxation and other particle-level dynamics, but do so with real-data in 1D or 1D-like scanning scenarios.
	\item Langevin-based models that operate in multi-dimensional FFP MPI, but do so in simulated scenarios.
	\item Langevin-based and beyond-Langevin models that both operate on multi-dimensional FFP MPI and show results on real data, but do so in a hybrid fashion, using a transfer function obtained from a calibration for preprocessing.
\end{enumerate}
Consequently, to the best of our knowledge, no method so far was able to provide 
fully model-based reconstructions (neither model transfer function (MTF) nor hybrid corrections)
of fully multi-dimensional FFP data while also being able to incorporate relaxation effects of the magnetic particles. 
This work contributes precisely to this vacuum in the MPI literature 
by considering a Debye model which gives rise to a relaxation adaption technique.
A preliminary version of this technique was shortly communicated in the conference proceeding~\cite{gapyak2026iwmpi}, 
but without details on the method or its analysis, which we provide now.
More specifically, the contributions of this work are:
\begin{enumerate}[1.]
	\item 
	Starting from a multi-dimensional variant of the Debye model, 
	we provide reconstruction formulae and show that the Debye model-based signal
	is the response of a linear time-invariant (LTI) system with exponential memory  
	applied to a Langevin model-based signal (theorem~\ref{theo:LTI}).
	\item 
	We derive, upon discretization in time, from the discrete-time LTI representation a first-order linear recurrence 
	of the signal data (equation \eqref{eq:relaxation:correction}), which we leverage to devise a \textit{relaxation adaption} step.
	The relaxation-adapted Debye signal data corresponds to Langevin model-based signal data
	which is consequently represented by the MPI core operator derived in earlier work from the Langevin theory. 
	Thus, we show that first order Debye relaxation models can be incorporated within the Langevin framework 
	without redesign of the reconstruction method, but only at the price of the adaption correction step, 
	which is of order $\mathcal{O}(L)$ where $L$ is the amount of time-samples collected during a scan.
	\item 
	Using our method which combines the relaxation-adaption with MoBiT-2S we obtain
	fully 2D, fully model-based (no model transfer function) reconstructions of good quality
	from real 2D MPI data.
\end{enumerate}

\section{Methods}

We begin by introducing the preliminaries of Langevin model-based MPI in section~\ref{sec:langevin}. 
We then introduce the multi-dimensional Debye model in section~\ref{sec:debye} and derive 
the representation of the scan signal $\svect$
as the response of an LTI system with exponential memory applied to a Langevin model-based input signal.
From this representation we derive the relaxation adaption step. 
Finally, in section~\ref{sec:reco:algorithm} we describe the stages 
of the proposed reconstruction algorithm.

\subsection{The Langevin model}\label{sec:langevin}

In MPI a liquid (tracer) containing magnetic nanoparticles is injected into the target specimen. 
These nanoparticles spread within the specimen (e.g., through the blood-vessel system of a small animal) 
and we are interested to image their spatial distribution.
The particle concentration is 
a function $\rho\colon \R^3 \to \R^{\geq 0}$ with compact support 
which represents the concentration of nanoparticles at each point in space.
These nanoparticles react to changes in a magnetic field and consequently, 
to image their distribution, a chosen time-depending magnetic field $\Hxt \in \R^3$ is applied. 
From now on, $\xvec \in \R^3$ indicates the position in space and $t \in \R$ is the time variable. 
As a result of the scan, one obtains the (background corrected) signal 
that can be modeled as~\cite{kluth2018mathematical}
\begin{equation}\label{eq:general:signal:s}
	\svect = \left [-\mu_0\frac{d}{dt} \int_{\mathbb{R}^3}R(\xvec )\Mxt\, d\xvec\right ] * \analogft ,
\end{equation}
where $\Mxt\in\mathbb{R}^3$ is the magnetization response of the particle distribution $\rho$, $R(\xvec )\in\mathbb{R}^{3\times 3}$ is the sensitivity profile of the receive coils, $\analogft$ is a vector kernel representing the analog filtering of the signal, and $\mu_0$ is a physical constant (magnetic permeability in vacuum).
Usually, the function $\analogft$ is known and can be factored out by a division in Fourier domain, so that the signal is represented more simply as
\begin{equation}\label{eq:signal:s}
	\svect = -\mu_0\frac{d}{dt} \int_{\mathbb{R}^3}R(\xvec )\Mxt\, d\xvec  .
\end{equation}
To proceed further, it is important to choose the model for the magnetization response $\Mxt$. 
According to the adiabatic magnetization Langevin model, the magnetization response $\Mxtad$ can be modeled as
\begin{equation}\label{eq:mag:langevin}
	\Mxtad = m\rho (\xvec )\Langevin\flexpar*{\Hhsat}\Hxtn
\end{equation}
where $m$ is the magnetic moment of a single particle, and
\begin{equation}\label{eq:hsat}
	\hsat = \frac{k_B T}{ \msat \frac{\pi}{6}d^3}
\end{equation}
is a constant encapsulating physical properties of the particles such as their temperature $T$, 
their magnetic saturation $\msat$, their diameter $d$ and the Boltzmann's constant $k_B$.
With the choice of model as in \eqref{eq:mag:langevin} we denote the corresponding signal following 
equation \eqref{eq:signal:s} by $\sadvec(t)$. We note that, although the real setup is always 3D,
with particle concentrations of the form $\rho(x,y,z) = \tilde{\rho}(x,y) \cdot \delta(z)$,
where $\delta(z)$ means the Dirac-$\delta$, and a scanning trajectory in the $xy$-plane
we have a 2D setup, cf.~\cite{marz2016model}. 
Thus the dimensions $n$ considered are $n=2$ or $n=3$.
The representation of $\sadvec(t)$ has been shown to be mediated by the 
so called MPI Core Operator~\cite{marz2016model} which we now recall.
First, let us consider the (matrix-valued) MPI Kernel 
\begin{align}\label{eq:kernel:matrix}
	K(\yvec ) &= \Langevin '(\norm{\yvec}_{2})\frac{\yvec\yvec^T}{\norm{\yvec}_{2}^2} + \frac{\Langevin (\norm{\yvec}_{2})}{\norm{\yvec}_{2}}\left [ \mathbb{I}- \frac{\yvec\yvec^T}{\norm{\yvec}_{2}^2} \right ] &
	& \text{for }\yvec\in\mathbb{R}^n\setminus \lbrace 0\rbrace ,
\end{align}
extended by $K(0)=\frac{1}{3}\mathbb{I}$ by continuity. The MPI Kernel in \eqref{eq:kernel:matrix} 
is the derivative $\frac{d}{d\yvec}[\Langevin (\norm{\yvec})\frac{\yvec}{\norm{\yvec}}]$ 
of the function $\Langevin (\norm{\yvec})\frac{\yvec}{\norm{\yvec}}$ in \eqref{eq:mag:langevin}. 
We denote with the symbol $K_h$ its resolution-rescaled (dilated) version $K_h (y) = \frac{1}{h}K(\frac{y}{h})$ 
for a resolution parameter $h$. Following the analysis in \cite{gapyak2023ffl3d}, the MPI kernel 
is a smooth and bounded kernel in $C^{\infty}_b (\R^n ;\R^{n\times n})$, 
where $C^{\infty}_b$ denotes smooth and bounded functions. 
If now $\rho$ is a particular particle concentration, we define the MPI Core Operator $A_h$ 
as the convolution operator on the target concentration using the MPI Kernel
\begin{align}\label{eq:coreoperator}
	A_h[\rho](x) &= (K_h * \rho ) (x) = \int_{\R^n} \rho (y) K_h (x-y)\, dy \, .
\end{align}
It has been shown in \cite{gapyak2023ffl3d} that $A_h$ in~\eqref{eq:coreoperator} 
is well defined for any $\rho\in L^1 (\R^n)$, where $L^1$ denotes the space of integrable functions,
and that $A_h\colon L^1 (\R^n) \mapsto C_b^{\infty} (\R^n,\R^{n\times n})$ is a bounded linear operator.
Finally, in FFP MPI the magnetic field applied $\Hxt$ is the superposition of a static \emph{selection field} 
$\Hvec_s(\xvec)$ and a dynamic \emph{drive field} $\Hvec_d(t)$. 
The selection field is usually the simple field $\Hvec_s (\xvec ) = G\xvec$ 
for some constant gradient $G\in\R^{n\times n}$ and the drive field can be represented as 
$\Hvec_d (t) = -G\rvect$, for a given scanning trajectory $\rvect$. 
Consequently, the magnetic field applied during a scan in FFP MPI is
\begin{equation}\label{eq:H}
	\Hxt = \Hvec_s (\xvec) + \Hvec_d (t) = G (\xvec - \rvect ).
\end{equation}
With the standard choice of $R(\xvec )$ (typically constant, $R(\xvec)\equiv R\in\R^{n\times n}$), 
the signal $\sadvec(t)$ corresponding to the Langevin model $\Mxtad$ is represented as
\begin{equation}\label{eq:reco:formula:langevin}
	\sadvec(t) = -\mu_0 \; m \; R \; A_{\hsat}[\rho] \flexpar*{G\rvect} \; G\vvect .
\end{equation}

\subsection{The Debye model and the relaxation adaption}\label{sec:debye}

With the Langevin model $\Mxtad$ in \eqref{eq:mag:langevin}, the magnetization of the particles $\Mxtad$ is a vector field parallel to 
$\Hxt$ for each point $t$ in time. Consequently, instantaneous re-alignment of the nanoparticles' momenta with the varying magnetic field $\Hxt$
is assumed. To model relaxation effects, we consider the Debye model~\cite{Croft2012relaxation} in a multi-dimensional scenario.
We propose the multi-dimensional Debye model defined via the ordinary differential equation
\begin{align}\label{eq:debye:ode}
	\frac{d\Mxt}{dt} &= -\frac{\Mxt - \Mxtad}{\tau}, &
	&\text{for} \qquad t > 0,
\end{align}
where $\Mxtad$ is the magnetization according to the Langevin model from \eqref{eq:mag:langevin}.
Before proceeding further we fix the following hypotheses on the objects under consideration:
\begin{itemize}
	\item[(P1)] the function $\rho$ is an integrable function, i.e., $\rho \in L^1(\R^n)$;
	\item[(P2)] the scanning trajectory is differentiable, $\rvec \in C^1([0;\infty); \R^n)$;
\end{itemize}
\begin{lemma}\label{lem:debye}
	With the assumptions (P1) and (P2) the Cauchy problem 
	\begin{align}\label{eq:debye:ivp}
		\frac{d\Mxt}{dt}  &= -\frac{\Mxt - \Mxtad}{\tau}, \quad t > 0, & \Mvec(\xvec,0) = \Mvec_0(\xvec),
	\end{align}
	for any initial state $\Mvec_0 \in L^1(\R^n; \R^n)$
	is well defined for almost every $\xvec \in \R^n$. 
	The magnetization $\Mxt$ is given by Duhamel's formula
	\begin{align}\label{eq:duhamel:M}
		\Mxt &= \e^{-\frac{t}{\tau}} \Mvec_0(\xvec) + \int\limits_0^t \frac{\e^{-\frac{(t-z)}{\tau}}}{\tau} \; \Madvec(\xvec,z) \; dz
	\end{align}
	and $\Mxt$ is integrable over $\R^n$ for every fixed $t \geq 0$.
\end{lemma}

\begin{proof}
	Since $\rvec$ is $C^1$, the applied magnetic field given by equation \eqref{eq:H} 
	with $G$ constant is differentiable once w.r.t. time $t$ and infinitely many times w.r.t the spatial 
	variable $\xvec$. Consider the definition of $\Mxtad$ in \eqref{eq:mag:langevin}. 
	The function $g(\yvec) = \Langevin (\norm{\yvec})\frac{\yvec}{\norm{\yvec}}$ is bounded and smooth away from zero
	with continuous extensions of $g$ and its derivative into $\yvec = 0$. 
	Consequently, since $\Hxt$ is $C^1$, then $g(\Hvec)\in C^1$ as a function of time and so is $\Mxtad$. 
	Because $\Hxt$ is smooth in $\xvec$, $g(\Hvec)\in C^1$ as a function of $\xvec$ and bounded.
	Hence $\Mxtad$ is integrable for every $t \geq 0$ since $\rho \cdot g(\Hvec)$ is. 
	For a given initial state $\Mvec_0 \in L^1(\R^n ; \R^n)$ we obtain therefore
	for almost every fixed $\xvec \in \R^n$ an initial value problem for an ordinary differential equation.
	The solution is found by variation of the constants and yields Duhamel's formula.   
	Since $\Mvec_0(\xvec)$ and $\Mxtad$ (for fixed $t$) are integrable functions, 
	so is $\Mxt$ by equation \eqref{eq:duhamel:M}. 
\end{proof}

\begin{theorem}\label{theo:LTI}
	Under the assumptions (P1) and (P2), the signal $\svect$ of equation \eqref{eq:signal:s}
	when employing the magnetization from the Debye model in equation \eqref{eq:debye:ode} 
	solves the initial value problem
	\begin{align}\label{eq:signal:ivp}
		\frac{d}{dt} \svect & = -\frac{1}{\tau}\svect + \frac{1}{\tau} \sadvec(t), \qquad t > 0, &
		\svec(0) & = \svec_0,
	\end{align}
	where $\sadvec$ is the signal as given by the Langevin model and the initial state  
	$\svec_0$ is consistent with the initial state of $\Mxt$ from equation \eqref{eq:debye:ivp}
	and the initial state of $\Mxtad$.
	The signal $\svect$ is given by Duhamel's formula
	\begin{align}\label{eq:duhamel:s}
		\svect &= \e^{-\frac{t}{\tau}} \svec_0 + \int\limits_0^t \frac{\e^{-\frac{(t-z)}{\tau}}}{\tau} \; \sadvec(z) \; dz\, .
	\end{align}
\end{theorem}

\begin{proof}
	As a result of lemma~\ref{lem:debye}, $\Mxt$ is integrable over $\R^n$ for every fixed $t \geq 0$. 
	Therefore, the signal, which involves the integral of $\Mxt$ over $\xvec$, is well defined.
	By interchanging integration and differentiation w.r.t. time $t$ we obtain 
	\begin{align}
		\svect &= -\mu_0 \int_{\R^n} R(\xvec) \frac{d}{dt} \Mxt\, d\xvec
		= -\mu_0 \int_{\R^n} R(\xvec) \left(-\frac{\Mxt - \Mxtad}{\tau}\right)\, d\xvec,
	\end{align}	
	where in the last step we have used the Debye model differential equation \eqref{eq:debye:ode}.
	The latter result can be rearranged to 
	\begin{align}\label{eq:intermediate:s}
		\svect &
		= -\frac{1}{\tau} \left(-\mu_0 \int_{\R^n} R(\xvec) \Mxt \, d\xvec  \right)
		+  \frac{1}{\tau} \left(-\mu_0 \int_{\R^n} R(\xvec) \Mxtad \, d\xvec \right).
	\end{align} 
	By differentiating equation \eqref{eq:intermediate:s} w.r.t. time $t$ 
	and reusing the definitions of the signal $\svect$ and the Langevin model signal $\sadvec(t)$
	we derive the differential equation 
	\begin{align}
		\frac{d}{dt} \svect &
		= -\frac{1}{\tau} \left(-\mu_0 \frac{d}{dt} \int_{\R^n} R(\xvec) \Mxt\, d\xvec  \right) 
		+  \frac{1}{\tau} \left(-\mu_0 \frac{d}{dt} \int_{\R^n} R(\xvec) \Mxtad \, d\xvec \right) \\
		&= -\frac{1}{\tau}\svect + \frac{1}{\tau} \sadvec(t).
	\end{align} 
	The initial value follows from equation \eqref{eq:intermediate:s}
	\begin{align}
		\svec(0) &= -\frac{1}{\tau} \left(-\mu_0 \int_{\R^n} R(\xvec) \Mvec_0(\xvec) \, d\xvec  \right) 
		+ \frac{1}{\tau} \left(-\mu_0 \int_{\R^n} R(\xvec) \Madvec(\xvec,0) \, d\xvec  \right)
		=: \svec_0.
	\end{align}
	Variation of the constants yields the signal description in equation \eqref{eq:duhamel:s}.
\end{proof}

The initial value problem \eqref{eq:signal:ivp} constitutes the forward model of the signal 
accounting for relaxation in the continuous scenario. The forward model maps the Langevin model based signal $\sadvec$ 
to the Debye model based $\svec$ via formula \eqref{eq:duhamel:s}.
In the following we consider the corresponding inverse problem, i.e., from the measured signal $\svec$
we want to derive the corresponding Langevin model based signal $\sadvec$ 
under the assumption that the measured signal $\svec$ obeys the Debye model.
This means that we want to solve the following Volterra integral equation of the first kind
\begin{align}\label{eq:volterra:sad}
	\int\limits_0^t \frac{\e^{-\frac{(t-z)}{\tau}}}{\tau} \; \sadvec(z) \; dz  &=  \svect - \e^{-\frac{t}{\tau}} \svec(0)
\end{align}
for $\sadvec$.
In MPI applications the signal $\svect$ is discretely sampled at specific time samples.
We now consider a scan with $L \in \N$ samples in a time window of interest $[0,T]$ 
at discrete times $t_k = k \Deltat$ for $k = 0, 1, \dots , L$ with time step size $\Deltat = T/L$.
Given the discrete measured data $\sveck = \svec(t_k )$
the goal is to calculate $\sadveck \approx \sadvec(t_k)$. 
\begin{align}
	\svec (t_n ) - \e^{-\frac{n\Deltat}{\tau}}\svec(0) &=  \int_{0}^{t_n} \frac{\e^{-\frac{(t_n-z)}{\tau}}}{\tau} \sadvec(z)\, dz 
	= \sum_{k=1}^{n}\int_{t_{k-1}}^{t_k} \frac{\e^{-\frac{(t_n-z)}{\tau}}}{\tau} \sadvec (z)\, dz. \label{eq:stn:continuous}
\end{align}
We use the quadrature approach of~\cite{atkinson1967quadrature} and 
approximate $\sadvec(t)$ by a constant $\sadveck$ on the time interval $[t_{k-1},t_k]$.
This yields the following approximations of the integrals in \eqref{eq:stn:continuous} 
\begin{align}
	& \int_{t_{k-1}}^{t_k} \frac{ \e^{-\frac{(t_n-z)}{\tau}}}{\tau} \sadvec (z)\, dz
	\approx  \int_{t_{k-1}}^{t_k}\frac{e^{-\frac{(t_n-z)}{\tau}}}{\tau}\, dz \; \sadveck
	= \e^{-\frac{(n-k)\Deltat}{\tau}} \left( 1-\e^{-\frac{\Deltat}{\tau}} \right) \; \sadveck . \label{eq:quadrature}
\end{align}
where the remaining integral was resolved symbolically. 
With that and by setting $\alpha\coloneq e^{-\frac{\Deltat}{\tau}}$ 
we obtain the following discrete version of the integral equation \eqref{eq:volterra:sad}
\begin{align}\label{eq:volterra:sad:discrete}
	(1-\alpha)\sum_{k=1}^{n} \alpha^{n-k}\sadveck &= \svec_n - \alpha^n \svec_0, & 
	& \text{for} \qquad n \in \{1,\ldots,L\}.
\end{align}
By performing a step of Gaussian elimination, where we subtract 
from equation \eqref{eq:volterra:sad:discrete} for the $n$-th time step
$\alpha$ times the equation \eqref{eq:volterra:sad:discrete} for $(n-1)$-th time step, we arrive at 
\begin{align}\label{eq:relaxation:correction}
	\sadvecn &= \frac{\svec_n - \alpha \svec_{n-1}}{1-\alpha} &
	& \text{for} \quad n \in \{1,\ldots,L\}.
\end{align}
As a consequence of \eqref{eq:relaxation:correction} by knowing $\alpha$
the corresponding Langevin model based data $\sadvecn$ can be calculated
and fed into an MPI reconstruction algorithm which is based on the Langevin model.
This is the topic of section~\ref{sec:reco:algorithm}.

Before going into the details of the reconstruction algorithm we note that 
the inverse problem \eqref{eq:volterra:sad} is ill-posed and thus 
the discrete version \eqref{eq:volterra:sad:discrete} is expected to be ill-conditioned.
However, when working with the real data used in the experiments (see section~\ref{sec:experiments})
the relaxation time $\tau$ varies between $10^{-6}$ and $5 \cdot 10^{-5}$ 
while the times step size is fixed at $\Deltat = 4 \cdot 10^{-7}$. 
With that the parameter $\alpha$ varies between $0.6703\ldots$  
and $0.9920\ldots$, but stays away from $1$. Based on the row sum norm the condition number 
$\kappa_{\infty}(B) = \|B\|_{\infty} \|B^{-1}\|_{\infty}$ of the matrix $B$
which describes the discrete problem is $\kappa_{\infty}(B) = (1-\e^{-T/\tau})(1+\alpha)/(1-\alpha)$.
Given the range of $\alpha$ the condition number varies between $5$ and $251$ and is rather mild.
The corresponding error amplification can thus be handled by the regularization applied 
in later stages of the overall algorithm.
In scenarios where $\Deltat \ll \tau $ the discrete problem needs to be regularized 
since the difference operator of the current approach \eqref{eq:relaxation:correction} 
approximates the differential operator $\tau \frac{d}{dt} + 1$ of equation \eqref{eq:signal:ivp}
as $\alpha$ tends to $1$ with $\Deltat \to 0$  
and $\kappa_{\infty}(B)$ would become essentially larger. 
Studying such scenarios is a topic of future research.

\subsection{Reconstruction algorithm}\label{sec:reco:algorithm}

In what follows we describe the case of 2D reconstructions, however, 
the steps are straightforwardly generalizable to the 3D scenario.
As explained in the previous section the data is sampled at discrete times $t_k$ within the time window $[0,T]$.
In this way, the input data for a reconstruction algorithm are the sampling of the signal $\sveck = \svec (t_k )$, 
the trajectory positions $\rveck = \rvec (t_k )$ and the velocities of the trajectory $\vveck = \vvec (t_k )$ of the FFP. 
Moreover, we reconstruct $\rho$ within a region of interest -- the field of view (FOV) -- which is a 
box-shaped set $\Omega = [a,b]\times [c,d]$ which contains the trajectory $\rvect$. 
The reconstruction is performed using an $N_x\times N_y$ Cartesian grid of the FOV $\Omega$,
thus functions of the space variable $\xvec \in \Omega$ are approximated by grid functions 
which are constant on a single grid cell (pixel).

When using the Langevin model as per \eqref{eq:reco:formula:langevin} a two-stage reconstruction 
has been proposed and further developed by the authors~\cite{marz2016model,gapyak2022mdpi,gapyak2023multipatch,gapyak2025trajectoryindep,maerz2025higherorder}, 
which involves the following two steps:
\begin{enumerate}[i)]
	\item MPI Core Stage: given the input data $\svec_n$, $\rvec_n$ $\vvec_n$,  
	the MPI Core Response $A_{\hsat}[\rho ]$ is reconstructed based on the relation in \eqref{eq:reco:formula:langevin}. 
	The output of this stage is the $\R^{2 \times 2}$-valued grid function $A \in \R^{N_x\times N_y \times 2\times 2}$ 
	which approximates the MPI Core Response on the grid.
	\item Deconvolution Stage: once $A\approx A_{\hsat}[\rho ]$ is reconstructed, the relation $K_{\hsat}*\rho = A_{\hsat}[\rho]$ from 
	equation \eqref{eq:coreoperator} is used.
	Regularized deconvolution of the MPI Core Response $A$ with the kernel $K_{\hsat}$
	yields the final reconstruction $\rho$ as grid function $\rho \in \R^{N_x\times N_y}$. 
\end{enumerate}
This two stage algorithm reconstructs $\rho$ based on the Langevin model.
In section \ref{sec:debye} we have seen that, when assuming the multi-dimensional Debye model in \eqref{eq:debye:ode}, 
the measured data $\svect$ is not the same as $\sadvec(t)$ which would correspond to the Langevin model.
Therefore, this two stage algorithm would not be directly applicable to the inputs $\svec_k$, $\rvec_k$ and $\vvec_k$.
However, with equation \eqref{eq:relaxation:correction} we can perform a relaxation adaption step
which provides the corresponding Langevin model data $\sadveck$ from the measured data $\sveck$.
With the adapted inputs $\sadveck$, $\rvec_k$ and $\vvec_k$ the algorithm above can be reused. 
Consequently, a reconstruction algorithm for the multi-dimensional Debye model can be formulated 
as a three-stage reconstruction algorithm as follows:
\begin{enumerate}[i)]
	\item Relaxation Adaption Step: given the measured data $\sveck$, compute $\sadveck$ by equation \eqref{eq:relaxation:correction}. 
	\item MPI Core Stage: reconstruct the MPI Core Response $A \approx A_{\hsat}[\rho]$ using the inputs $\sadveck$, $\rvec_k$ and $\vvec_k$.
	\item Deconvolution Stage: reconstruct $\rho$ by deconvolving the MPI Core Response $A$ with the kernel $K_{\hsat}$.
\end{enumerate}
Note that by the simplicity of equation \eqref{eq:relaxation:correction} the relaxation adaption step is 
itself computationally inexpensive, while the computational costs of the MPI core stage and the deconvolution stage are unaffected.
With this three-stage algorithm we propose a reconstruction scheme that incorporates the relaxation of the nanoparticles 
according to the Debye model 
with small additional computational cost of order $\mathcal{O}(L)$, where $L$ is the sample size.

In the following paragraphs we provide more details on the MPI Core Stage and the Deconvolution Stage.

\paragraph{MPI Core Stage.}
The aim of the MPI Core Stage is the reconstruction of a discrete version 
$A \in \R^{N_x\times N_y \times 2\times 2}$ of the MPI Core Response on a $N_x\times N_y$ grid, 
given the inputs $\sadveck$, $\rvec_k$ and $\vvec_k$ along the trajectory of the FFP.  
Following the variational approach of \cite{gapyak2022mdpi} $A$ is reconstruct by energy minimization: 
\begin{align}\label{eq:first:step}
	A &= \arg\min\limits_{\hat{A}}\biggl\lbrace 
	\underbrace{\frac{1}{L} \sum\limits_{k=1}^L \left\lVert \sadveck - \interpolator \bigl [\hat{A}\bigr ](\rvec_k ) \vvec_k \right\rVert_{2}^2 }_{\coloneq F[\hat{A}]}
	+ \gamma\regularizer_C [\hat{A}]
	\biggr\rbrace .
\end{align}
Here the data fidelity term $F[\hat{A}]$ is motivated by \eqref{eq:reco:formula:langevin}
but involves an interpolation scheme $\interpolator$ since the evaluation points $\rveck$ do 
not coincide with grid cell centers.
$\regularizer_C$ denotes the chosen regularizer and $\gamma >0$ is the regularization weight. 
Following \cite{maerz2025higherorder} we employ cosine interpolation and use the $\ell^2$-norm of the Bi-Laplacian 
as regularizer $\regularizer_C$. We refer the interested reader to \cite{gapyak2022mdpi} for details on the computation 
of the Euler-Lagrange equations associated with the problem \eqref{eq:first:step} and to \cite{maerz2025higherorder} 
for details on the choice and effect of the regularizer $\regularizer_C$. 
We point out that $A \in \R^{N_x\times N_y \times 2\times 2}$ in \eqref{eq:first:step} is obtained by solving 
the Euler-Lagrange equation with the Conjugate Gradient (CG) method.

\paragraph{Deconvolution Stage and multi-kernel deconvolution.}
The aim of the Deconvolution Stage is to retrieve the target distribution $\rho$ given the output $A$ of the MPI Core Stage. 
By equation \eqref{eq:coreoperator} we have the relation $K_{\hsat}*\rho = A_{\hsat}[\rho ]$.
It has been shown in \cite{marz2016model} that the trace contains all the necessary information for the reconstruction, 
i.e., the relation $\kappa_{\hsat}*\rho = \mathrm{tr}\left( A_{\hsat}[\rho] \right)$ for the scalar-valued 
trace kernel $\kappa_{\hsat} (\yvec )$ given by
\begin{align}
	\kappa_{\hsat} (\yvec ) &= \mathrm{tr}\left( K_{\hsat}(\yvec) \right) = \frac{1}{\hsat }\kappa \left(\frac{\yvec }{\hsat}\right), &
	\kappa (\yvec ) &= \Langevin ' (\norm{\yvec }_2 )+ \frac{\Langevin (\norm{\yvec}_2)}{\norm{\yvec}_2} (n-1).
\end{align}
It was shown in \cite{gapyak2023ffl3d} that the convolution with $\kappa_{\hsat}$ defines an injective operator. 
This motivates the deconvolution with the scalar valued kernel. 
In practice, the solution of this ill-posed problem is formulated as a minimization problem of the form
\begin{align}\label{eq:dec:min}
	\rho &= \arg\min_{\hat{\rho}}\left\lbrace \lVert\kappa_{\hsat} * \hat{\rho} - u\rVert_2^2 + \lambda \regularizer_D [\hat{\rho}]\right\rbrace ,
\end{align}
where $\regularizer_D$ is a chosen regularizer, $\lambda >0$ is the parameter controlling the strength of the regularization 
and $u = \mathrm{tr}(A)$ is the trace of the MPI Core Response reconstructed in the Core Stage. 
Although the trace kernel has nice properties (such as positivity and radial symmetry) 
the corresponding deconvolution problem employs only the diagonal entries of the matrices of the matrix-field $A$. 
Because of the challenging nature of the reconstruction problem in MPI and the noisiness of real data, 
the inclusion of additional information, when available, seems favourable. 
In particular, in the case that $\svect$, $\rvect$ and $\vvect$ are sets spanning $\mathbb{R}^2$, 
then all entries of $A$ are reconstructed. 
Consequently, without considering the trace of $A$, one can think about considering the whole matrix-valued 
deconvolution problem $A = K_{\hsat}*\rho$. 
This is equivalent to consider the multi-kernel deconvolution problem $K_{\hsat}^{i,j}*\rho = A_{\hsat}^{i,j}$ 
where the indexes $i,j$ denote the $(i,j)$-th entry of the $n\times n$ matrix-field $A$ and of the matrix kernel $K_{\hsat}$. 
Mathematically, this is justified: because convolution with $\kappa_{\hsat} = \mathrm{tr}(K_{\hsat})$ 
is injective, so must be convolution with $K_{\hsat}$.
The idea of using the entries of the MPI Core Response has been introduced in \cite{maerz2022icnaam} 
but was coupled with the use of traditional regularizers ($\ell^2$-norm and TV-Smooth regularizers). 

\begin{algorithm}[t]
	\caption{Proposed three-stage reconstruction algorithm for Debye model based reconstruction.
	}\label{alg:overall}
	
	\textbf{Input}: measured data $\sveck$ with corresponding locations $\rveck$ and velocities $\vveck$, 
	relaxation parameter $\tau$,
	regularization parameter $\gamma$, $\nu_0$, $n_{\mathrm{it}}$,
	and grid size parameters $N_x,N_y$.\\
	\textbf{Output}: reconstructed particle concentration $\rho \in \R^{N_x \times N_y}$.
	
	\begin{enumerate}[1.]
		\item Relaxation Adaption Step:
		\begin{algorithmic}
			\STATE $\alpha\gets \exp{(-\frac{\Delta t}{\tau})}$;
			\FOR{$k=1$ \TO $k=L$} 
			\STATE $\sadveck \gets \frac{\sveck -\alpha\svec_{k-1}}{1-\alpha}$;
			\ENDFOR
		\end{algorithmic}
		\item MPI Core Stage:
		\begin{algorithmic}
			\STATE $(G,b) \gets$ setup$(\sadvec,r,v,\gamma,N_x,N_y)$; \hfill\COMMENT{setup $G\vect(\hat{A})=b$ from E-L equations of \eqref{eq:first:step}}
			\STATE $\vect(\hat{A}) \gets $ ConjGrad$\left(G,b\right)$; \hfill\COMMENT{solve $G\vect (A)=b$}
			\STATE $A \gets $ inverseDCT$( \hat{A} )$; \hfill\COMMENT{$\hat{A} =$ DCT$(A)$}
		\end{algorithmic}
		\item Deconvolution Stage:
		\begin{algorithmic}
			\STATE $C\gets\sum\limits_{i=1}^{2}\sum\limits_{j=1}^{2}\alpha_{i,j}^* \alpha_{i,j}$; \hfill\COMMENT{$\alpha_{i,j}$ matrix representation of $K_h^{i,j}*\rho$}
			\STATE $d\gets\sum\limits_{i=1}^{2}\sum\limits_{j=1}^{2}\alpha_{i,j}^*\vect(A^{i,j})$; \hfill\COMMENT{$A^{i,j} \in \R^{N_x \times N_y}$ components of $A$}
			\STATE $\rho_2^0\gets 0 $;
			\FOR{$k=0$ \TO $k=n_{\mathrm{it}}$} 
			\STATE $\rho_1^{k+1}\gets$ ConjGrad$\left (C + \nu_k\mathrm{Id}\, , d + \nu_k \rho_2^k \right )$; 
			\STATE $\sigma_{k+1} \gets $ Noise-Estimator$(\rho_1^{k+1})$;
			\IF {$k=0$}
			\STATE $\lambda \gets\nu_0 \cdot\sigma_1^2$
			\ENDIF
			\STATE $\tilde{\rho}_2^{k+1}\gets$Denoiser$\left (\tilde{\rho}_1^{k+1}\, ,\sigma_{k+1}\right ) $;
			\STATE $\nu_{k+1}\gets \lambda /\sigma_{k+1}^2$;
			\STATE $k\gets k+1$;
			\ENDFOR
			\RETURN $\tilde{\rho}_2^{k+1}$
		\end{algorithmic}
	\end{enumerate}
\end{algorithm}

In what follows, we integrate the Plug-and-Play technique of the recent work \cite{gapyak2025trajectoryindep}
to provide a framework for the multi-kernel deconvolution problem 
arising when one wants to use all the entries of the MPI Core Kernel.
Plug-and-Play algorithms~\cite{venkatakrishnan2013pnp} in general integrate traditional splitting schemes 
and the usage of Gaussian denoisers which can be machine-learning-based~\cite{Zhang2022pnp}.
We start the discussion by formulating the multi-kernel deconvolution problem as the minimization problem
\begin{align}\label{eq:multiker:dec:min}
	\rho &= \arg\min_{\hat{\rho}}\left\lbrace \sum_{i=1}^{n}\sum_{j=1}^{n}\lVert K_{\hsat}^{i,j} * \hat{\rho} - A^{i,j}\rVert_2^2 + \lambda \regularizer_D [\hat{\rho}]\right\rbrace ,
\end{align}
where $A^{i,j}$ is the $(i,j)$-th entry of the reconstructed MPI Core Response. 
Following \cite{gapyak2025trajectoryindep}, we formulate the equivalent constrained minimization problem
\begin{align}\label{eq:multiker:dec:min:2}
	\rho &= \arg\min_{\rho_1 ,\, \rho_2}
	\left\lbrace\sum_{i=1}^{n}\sum_{j=1}^{n}\lVert K_{\hsat}^{i,j} * \rho_1 - A^{i,j}\rVert_2^2 + \lambda \regularizer_D [\rho_2]\right\rbrace &
	& \text{s.t.} \quad \rho_1 = \rho_2\, .
\end{align}
To solve \eqref{eq:multiker:dec:min:2} we consider half quadratic splitting, i.e., we consider the Lagrangian 
\begin{align}
	\mathscr{L}_{\nu}(\rho_1 , \rho_2) &= \lVert\kappa_h * \rho_1 - u\rVert_2^2 + \lambda \regularizer_D [\rho_2] + \nu\lVert \rho_1 - \rho_2\rVert_2^2 ,
\end{align}
with multiplier $\nu >0$, and minimize it by minimizing it iteratively in the two variable $\rho_1$ and $\rho_2$. 
The resulting scheme is
\begin{align}
	\rho_1^{k+1} & = \arg\min_{\rho_1}
	\left\lbrace\sum_{i=1}^{n}\sum_{j=1}^{n}\lVert K_{\hsat}^{i,j} * \rho_1 - A^{i,j}\rVert_2^2 +  \nu_k\left\lVert \rho_1-\rho_2^k \right\rVert_2^2\right\rbrace \label{eq:dec:sub:tik0}\\
	\rho_2^{k+1} & = \arg\min_{\rho_2}
	\left\lbrace\lambda \regularizer_D [\rho_2 ]+\nu_k\left\lVert \rho_1^{k+1}-\rho_2 \right\rVert_2^2 \right\rbrace , \label{eq:dec:sub:den:gauss}
\end{align}
where we have allowed the multiplier $\nu$ to vary at each iteration, a choice that yields the parameters 
$\nu_k$ in \eqref{eq:dec:sub:tik0}. Mathematically, the term in \eqref{eq:dec:sub:den:gauss} is the regularized 
denoising (inversion of the identity with) with assumption of Gaussian noise with noise level $\sqrt{\lambda /\nu_k}$. 
The idea behind Plug-and-Play approaches \cite{venkatakrishnan2013pnp,Zhang2022pnp} 
is that any Gaussian denoiser can be plugged in to address the denoising subproblem in equation \eqref{eq:dec:sub:den:gauss}. 
Here the denoising capabilities of modern machine learning based Gaussian denoisers 
can be leveraged~\cite{Zhang2022pnp}. 
Consequently, one can formulate the final deconvolution algorithm as the splitting scheme
\begin{align}
	\rho_1^{k+1} & = \arg\min_{\rho_1}
	\left\lbrace\sum_{i=1}^{n}\sum_{j=1}^{n}\lVert K_{\hsat}^{i,j} * \rho_1 - A^{i,j}\rVert_2^2 +  \nu_k\left\lVert \rho_1-\rho_2^k \right\rVert_2^2\right\rbrace \label{eq:dec:sub:tik}\\
	\rho_2^{k+1} & = \mathrm{Denoiser}\left (\tilde{\rho}_1^{k+1}\, , \sqrt{\lambda /\nu_k}\right ) , \label{eq:dec:sub:den}
\end{align}
where $\mathrm{Denoiser}$ denotes a trained denoiser and $\tilde{\cdot}$ represents the reshaping operator 
that turns vectors into 2D $N_x\times N_y$ arrays. Following \cite{gapyak2025trajectoryindep} 
for the scalar deconvolution Plug-and-Play algorithm we use the 
publicly available and pre-trained \emph{deep denoiser prior} from \cite{Zhang2022pnp} in a zero-shot fashion,
i.e., without any re-training or fine-tuning.
Further, we employ the parameter selection scheme from \cite{gapyak2025trajectoryindep}. 
The difference with the algorithm in \cite{gapyak2025trajectoryindep} lays 
in a different data fidelity term and thus a different subproblem given by equation \eqref{eq:dec:sub:tik}. 
The subproblem \eqref{eq:dec:sub:tik}, however, can be approached as before by solving the corresponding linear Euler-Lagrange equations. 
For more details on the implementation of the algorithm, we refer the interest reader to \cite{gapyak2025trajectoryindep}. 
The consequence of the parameter selection scheme from \cite{gapyak2025trajectoryindep} is that 
the only parameters needed for the scheme are the 
starting parameter $\nu_0$ and the number of iterations $n_{\mathrm{it}}$.
The steps of the proposed three-stage algorithm are summarized in algorithm~\ref{alg:overall}.

We remark that the multi-kernel data fidelity term in equation \eqref{eq:multiker:dec:min}
provides flexibilty: above we have use it to include all the information available,
but it can be adapted to reconstruct $\rho$ when only partial information is at our disposal~\cite{maerz2022icnaam}. 
To clarify this point, one could consider the data fidelity term
\begin{equation}\label{eq:multiker:dec:min:partial}
	\sum_{i=1}^{n}\sum_{j=1}^{n}\beta_{i,j} \lVert K_{\hsat}^{i,j} * \hat{\rho} - A^{i,j}\rVert_2^2 
\end{equation}
where $\beta_{i,j}$ is 1 if $A^{i,j}$ is available, and 0 if $A^{i,j}$ is not available. 
A use case for this is when only data of one recording channel is available or reliable: 
if only the $x$-component of the signal $\svect$ is available, 
then only the entries $A^{1,1}$ and $A^{1,2}$ can be reconstructed in the Core Stage 
and only those can be used for the deconvolution. 
In particular, in this case, the deconvolution using the trace is not possible, 
as only the first diagonal entry of $A$ can be reconstructed.

\section{Experiments}\label{sec:experiments}

\begin{figure}[t]
	\centering
	\begin{subfigure}[t]{0.20\textwidth}
		\includegraphics[width=\linewidth]{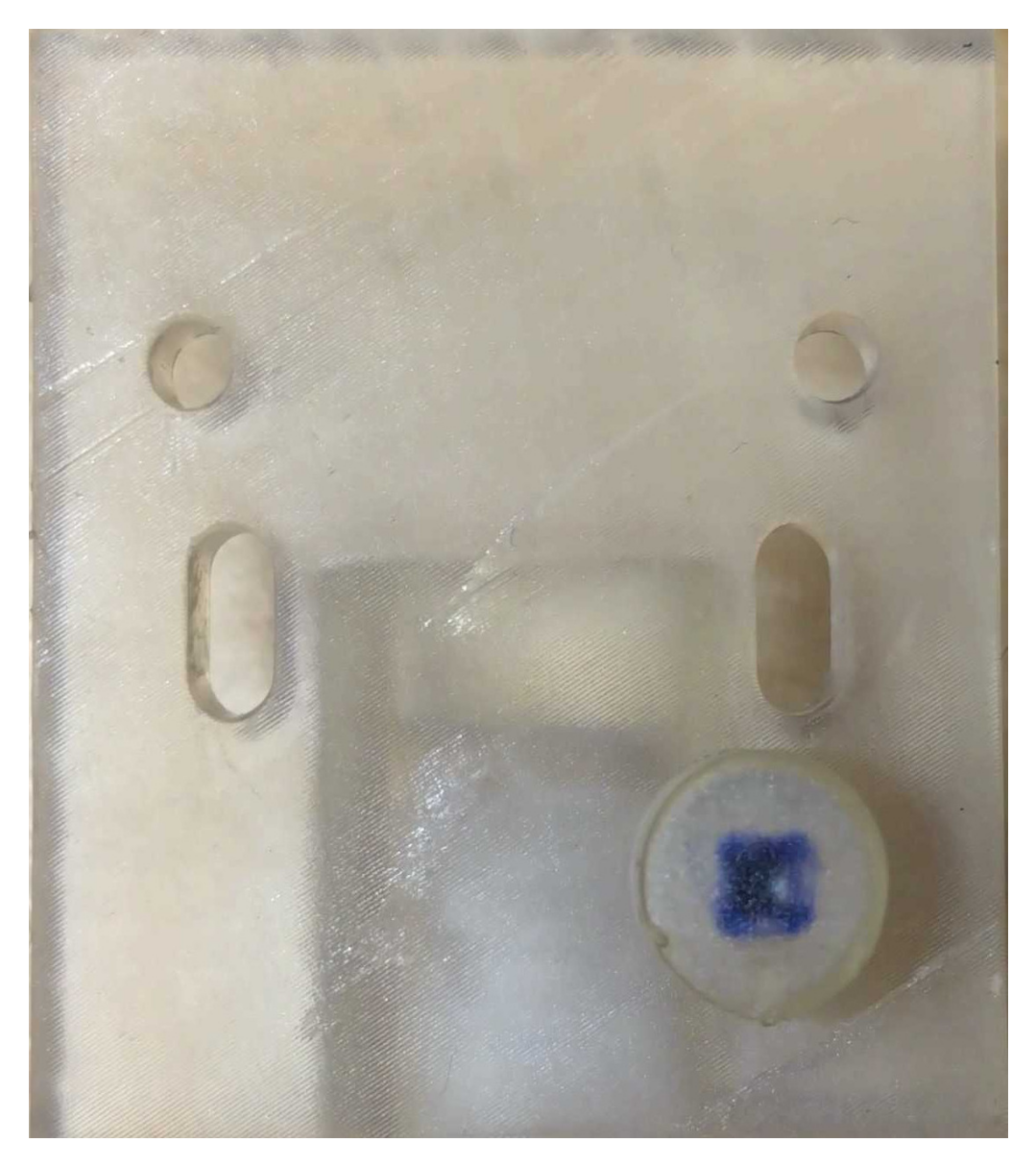}
		\caption{Dot}
		\label{subfig:dot}
	\end{subfigure}
	\begin{subfigure}[t]{0.20\textwidth}
		\includegraphics[width=\linewidth]{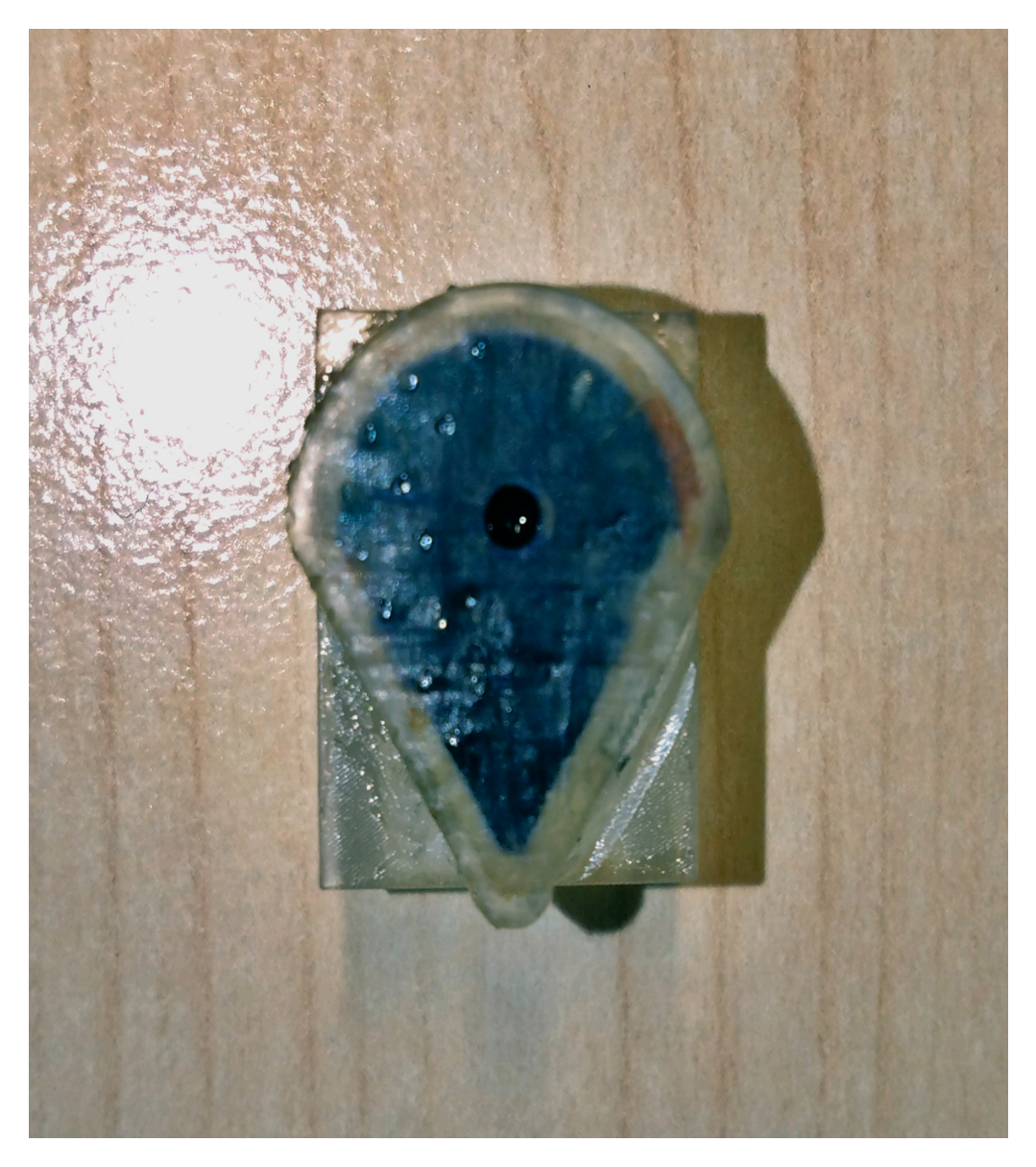}
		\caption{IceCream}
		\label{subfig:icecream}
	\end{subfigure}
	\begin{subfigure}[t]{0.20\textwidth}
		\includegraphics[width=\linewidth]{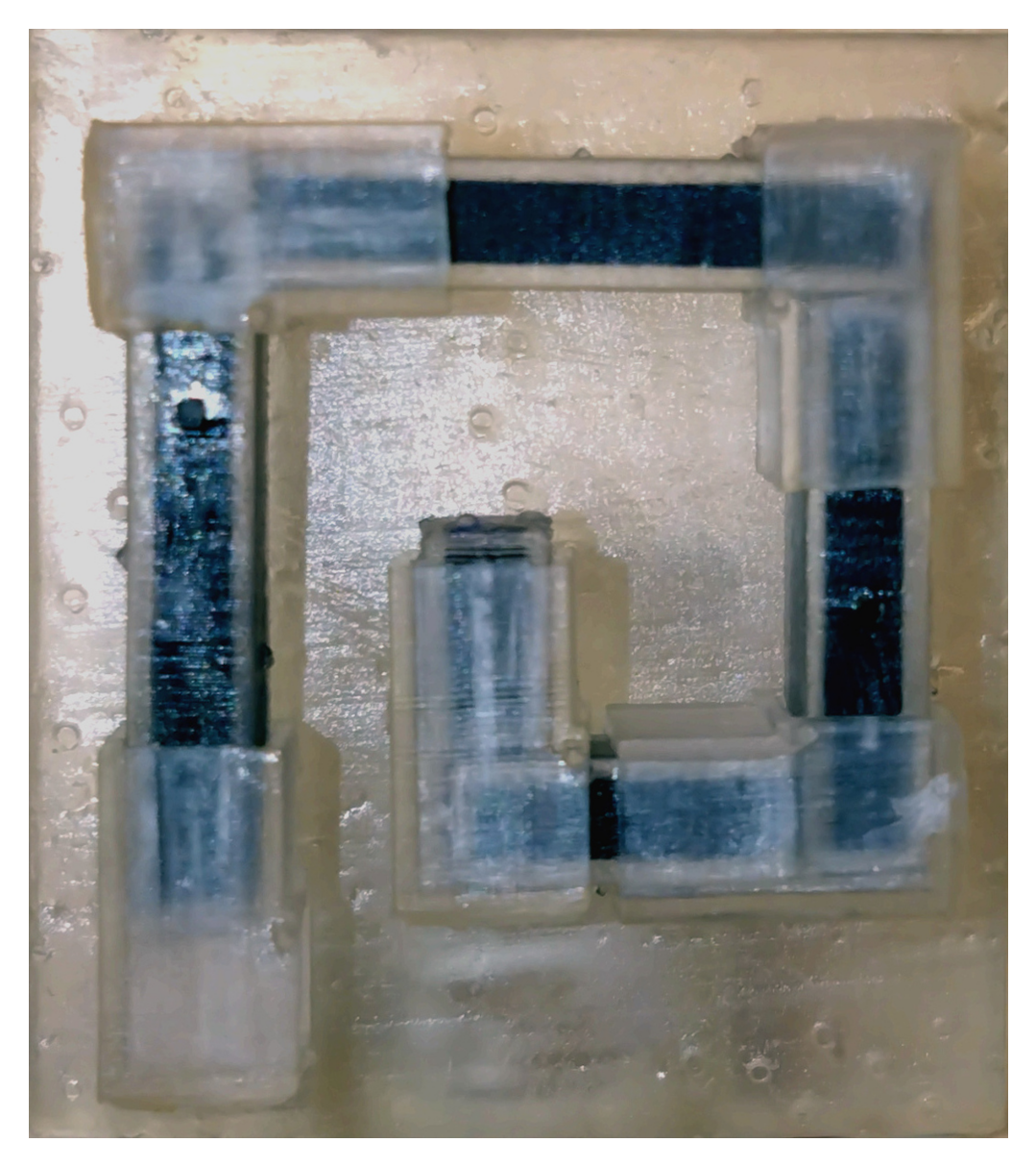}
		\caption{Snail}
		\label{subfig:snail}
	\end{subfigure}
	\caption{Pictures of 3 of the phantoms that are scanned in the real dataset~\cite{knopp2024equilibriumdata}. Reproduced from~\cite{knopp2024equilibriumdata}, \href{https://creativecommons.org/licenses/by/4.0/}{CC BY 4.0}.}
	\label{fig:real:phantoms}
\end{figure}

In this section we show the potential of our method in a series of experiments. 
After discussing preprocessing, we first perform qualitative and quantitative evaluation 
in a simulated scenario, where the ground truth is available. 
Then, we study the effects of the porposed relaxation adaption 
on the real 2D MPI data contained in the 
dataset~\cite{knopp2024equilibriumdata}. The three phantoms considered from the dataset 
in~\cite{knopp2024equilibriumdata} are displayed in figure~\ref{fig:real:phantoms}.

\subsection{Preprocessing of real data.} 
When dealing with real MPI data, a series of preprocessing steps are usually necessary. 
We briefly explain these steps and underline one specific step involving so called transfer functions. 

\paragraph{SNR-thresholding}
Real MPI data are noisy and a signal-to-noise ratio (SNR) of each frequency of the signal is usually estimated from a series of empty scans. With this SNR values, one has to some extent an overview of what frequencies are reliable and which ones are unreliable. This SNR values are used in the so-called SNR-thresholding step: a threshold $\Theta_i$ is set for each channel $i$, and the frequencies of the signal whose SNR is below $\Theta_i$ are discarded.

\paragraph{High-pass filtering.}
A further preprocessing step usually involves applying high-pass filters (HP) that cut out low frequencies. In what follows we will apply the HP filter that cuts all frequencies below the frequency $\hpf$
\begin{equation}
	\hpf = \lfloor \max\lbrace f_x, f_y, f_x\rbrace * n\rfloor + 100 ,
\end{equation}
for a chosen parameter value $n$. The effect of the high pass filter is to cut slightly above the n-th harmonic of the frequencies of the scanning trajectory.

\paragraph{Analog filtering transfer function.}
A next preprocessing step involves the division of the Fourier-transformed signal with 
the transfer function $\analogft$ related to the analog filtering (AF) 
and described in \eqref{eq:general:signal:s}. This function is usually provided as part of the 
scanner's specific data acquisition process. We call this transfer function the AF-transfer 
function (AF-TF). 

\paragraph{Model transfer function.}
Finally, typically a second transfer function is involved.
We now describe it to clarify an important aspect of the reconstruction presented in this paper. 
The second transfer function used in preprocessing 
is the model-transfer function (M-TF) first described in~\cite{KnoppBiederer_etal2010}. 
It is obtained from calibration data. The procedure can be described as follows: a known particle 
concentration is scanned in numerous (and known) positions in the FOV; this concentration is usually 
small and shaped as a voxel, such that the scan of it corresponds to measuring the response of the 
scanner to $\delta$-impulses. This $\delta$-concentration is moved and scanned in all grid-cells 
within a chosen grid discretizing the FOV. At each grid  position k, a scan $\svec_{\mathrm{real}}^{k} (t)$ 
is obtained and stored. In a subsequent step, the same calibration procedure is simulated using 
the Langevin model, yielding a set of corresponding scans $\svec_{\mathrm{sim}}^{k} (t)$. Under the assumption 
that $\svec_{\mathrm{real}}^{k}  = \avec_M * \svec_{\mathrm{sim}}^{k} $, one then retrieves the M-TF $\avec_M$ in a 
least square fashion. In this way, the deconvolved data $\hat{\svec}_{\mathrm{real}}^{k}/\hat{\avec}_M$ is 
cleaned of those convolutional components that make the real data non obeying the Langevin model. 
This employment of calibration data to obtain a M-TF and to preprocess the data has been applied 
in previous publications that employed the Langevin model on the real data 
in~\cite{knopp2024equilibriumdata}. The advantage of this method, is that the potentially complex 
convolutional component of the real data is taken care of and that it brings the data closer to the scenario 
modeled by the Langevin model, thus allowing for the usage of Langevin-based reconstruction algorithms. 
On the other hand, the usage of an M-TF requires calibration data and consequently, reconstructions 
obtained with such a correction cannot be considered as purely model-based but rather as hybrid 
reconstruction. 

In this paper however, we \emph{do not} employ the M-TF. In particular, with the 
framework derived here from the Debye model we show reconstructions of the real 2D MPI data in a fully 
model-based scenario using the three step algorithm developed in section~\ref{sec:reco:algorithm} 
which includes the relaxation adaption step to MoBit-2S. An overview of the parameters used for the 
real data is given in table~\ref{tab:params}.

\subsection{Qualitative and quantitative evaluation on simulated data.}
We have produced five simulated phantoms immitating the phantoms contained in the real dataset. 
These phantoms have been obtained using characters from the font \emph{Liberation Sans Font Regular}
\footnote{Liberation Sans Font Regular is licensed under GNU general public license (GPL) and available at \href{https://www.1001fonts.com/liberation-sans-font.html}{https://www.1001fonts.com/liberation-sans-font.html}}. 
The simulations have been performed generating the ground truths and scanning them along the Lissajous trajectory 
\begin{equation}\label{eq:lissajous}
	\rvect = \flexpar*{A_x \cos \flexpar*{2\pi f_x t}, \, A_y \cos \flexpar*{2\pi f_y t}}
\end{equation}
with amplitudes $A_x = A_y = 0.012\ \si{\tesla}/\mu_0$ and frequencies $f_x = \frac{2.5}{102}\ \si{\mega\hertz}$, $f_Y = \frac{2.5}{96}\ \si{\mega\hertz}$. 
The repetition time (time of scan) is $652.8\ \si{\micro\second}$ and the sampling time step size 
is $\Deltat = 4\times 10^{-7}\ \si{\second}$, resulting in $L=1632$ samples. The gradient has been simulated 
to be $\mu_G = \diag (-1,-1,2)\ \si{\tesla\per\meter}$. For the simulation of the particles we have 
computed $\hsat$ in \eqref{eq:hsat} using a temperature of $T_B = 293\ \si{\kelvin}$, a saturation 
magnetization $\msat = 4.74\times 10^{5}\ \si{\joule\tesla\per\meter\cubed}$ and a core diameter 
of $d = 21\ \si{\nano\meter}$. Finally, we have simulated a Debye delay parameter 
$\tau_{\mathrm{GT}} = 5\times 10^{-6}\ \si{\second}$ and added additional white Gaussian noise  
to the simulated scans such that the SNR is of $40\ \si{\deci\bel}$. We reconstruct the target 
concentrations on a $50\times 50$ grid within the FOV $\Omega = [-A_x ,A_x]\times [-A_y , A_y]$ using the 
relaxation adaption step and the MoBiT-2S algorithm.

In order to reduce the amount of parameters to validate, we first fix the parameter $\gamma$ for the core stage 
in \eqref{eq:first:step}. To this aim, we have chosen to perform reconstruction with the relaxation adaption 
parameter $\tau$ matching the ground truth relaxation $\tau_{\mathrm{GT}}$; after this, we have performed the 
core stage for a set of parameters $\gamma$ in $\lbrace i\cdot 10^j\colon i = 0, \dots ,9\wedge j = -7, -6, -5 \rbrace$. 
Given the outcome $A_{\gamma}\in\mathbb{R}^{2\times 2\times 50 \times 50}$ of the core stage, taken their 
traces $u_{\gamma}=\mathrm{tr}A_{\gamma}$ and have computed the PSNR with the ground truth 
$u_{\mathrm{GT}} = \mathrm{tr}A_{\mathrm{GT}}$. The parameter $\gamma^{*}$ that maximizes the PSNR averaged across 
all phantoms has been selected as the fixed reconstruction parameter for the core stage. The average PSNR curve for 
the parameter $\gamma$ is plotted in figure~\ref{subgfig:exp1:gamma:psnr}. The maximum PSNR value corresponds 
to $\gamma^* = 7\cdot 10^{-7}$. In figure~\ref{subgfig:exp1:gamma:ssim} we have displayed also the average SSIM curve.

With $\gamma^*$ fixed, we now perform the full reconstruction varying the relaxation adaption 
parameter $\tau$ in the set $\lbrace 0 \rbrace \cup \lbrace i\cdot 10^j\colon i = 0, \dots ,9\wedge j = -7, -6, -5 \rbrace$ 
and the starting parameter $\nu_0$ for the deconvolution stage in \eqref{eq:dec:sub:tik} 
in $\lbrace 10^{-5},10^{-6},10^{-7},10^{-8},10^{-9}\rbrace$ and for $n_{\mathrm{it}}=10$ iterations. 
We have then computed the PSNR values between the reconstructions $\rho_{\tau , \nu_0 ,n_{\mathrm{it}}}$ 
with the ground truth $\rho_{\mathrm{GT}}$ and selected the parameters $(\tau^* ,\nu_0^* , n_{\mathrm{it}}^* )$ 
that maximize the average PSNR. The values of the average PSNR for varying $\tau$ and $\nu_0$ (but $n_{\mathrm{it}}$) 
are displayed in figure~\ref{subfig:exp1:taunu:psnr}. The analogous values for the average SSIM score are displayed 
in figure~\ref{subfig:exp1:taunu:ssim}.
\begin{figure}[t!]
	\centering
	\begin{subfigure}[t]{0.45\textwidth}
		\includegraphics[width=\linewidth]{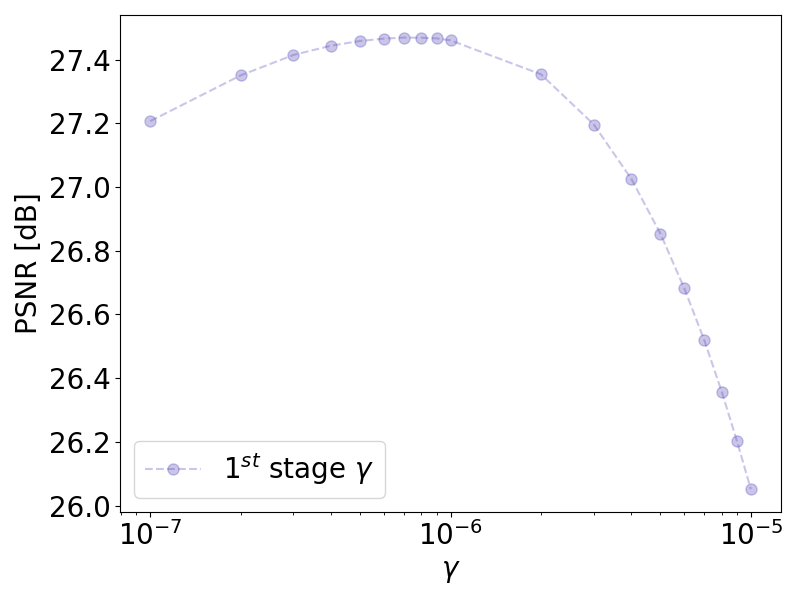}
		\caption{Average PSNR curve with adaption parameter $\tau = 5\cdot 10^{-6}$ coinciding with the simulated delay $\tau_{\mathrm{GT}}$, for a variety of regularization parameters $\gamma$ for the Core Stage. The maximum PSNR value corresponds to the parameter $\gamma^* = 7\cdot 10^{-7}$.}
		\label{subgfig:exp1:gamma:psnr}
	\end{subfigure}
	\hfil
	\begin{subfigure}[t]{0.45\textwidth}
		\includegraphics[width=\linewidth]{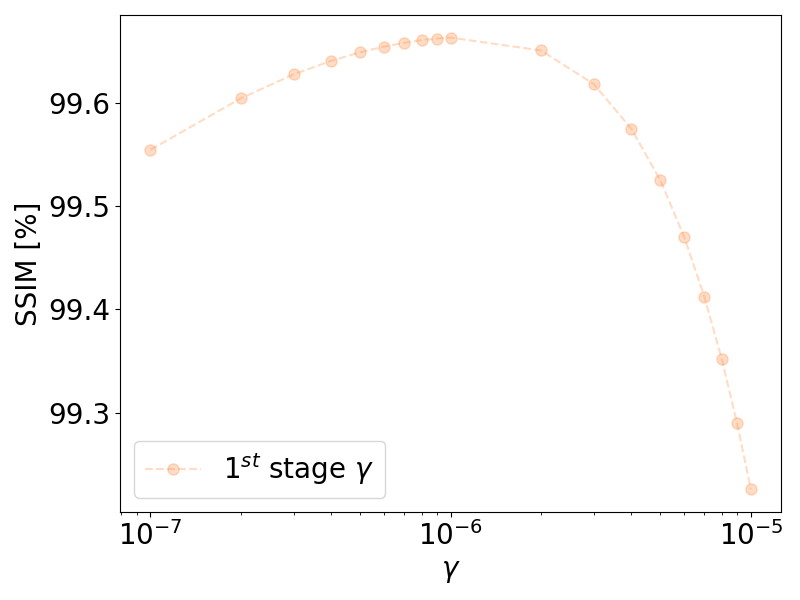}
		\caption{Average SSIM curve with adaption parameter $\tau = 5\cdot 10^{-6}$ coinciding with the simulated delay, for a variety of regularization parameters $\gamma$ for the Core Stage. }
		\label{subgfig:exp1:gamma:ssim}
	\end{subfigure}
	\par
	\begin{subfigure}[t]{0.45\textwidth}
		\includegraphics[width=\linewidth]{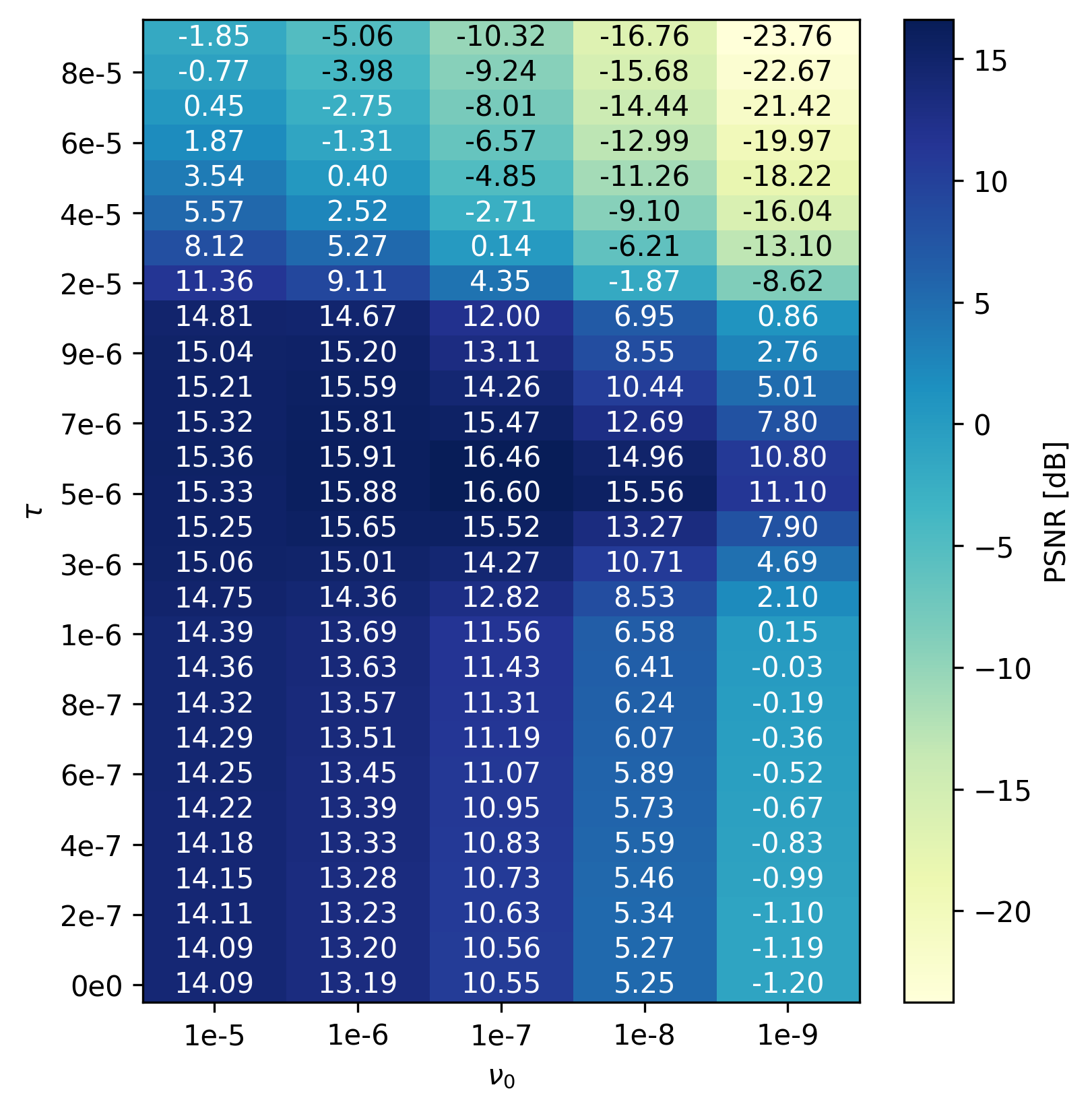}
		\caption{Average PSNR values of the final reconstruction for all combinations of relaxation correction parameters $\tau$ and Deconvolution Stage parameters $\nu_0$, with $\gamma = \gamma^*$ fixed for all reconstruction. In accordance with expectation, the absolute maximum value corresponds to the choice of $\tau = \tau_{\mathrm{GT}} = 5\cdot 10^{-6}$. The optimal deconvolution parameter is $\nu_0^* = 10^{-7}$ and $n_{\mathrm{it}}^* = 10$.}
		\label{subfig:exp1:taunu:psnr}
	\end{subfigure}
	\hfil
	\begin{subfigure}[t]{0.45\textwidth}
		\includegraphics[width=\linewidth]{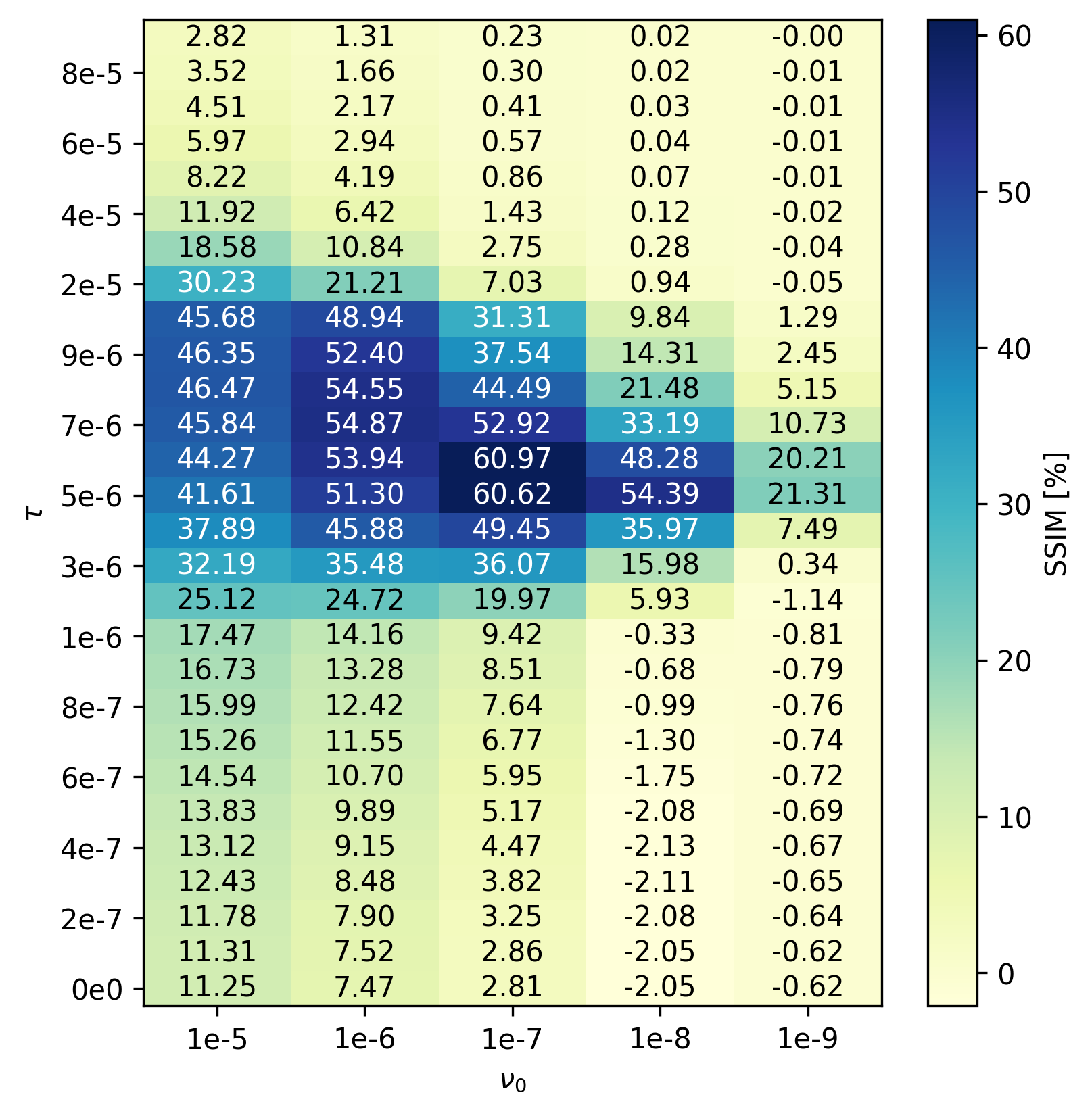}
		\caption{Average SSIM values of the final reconstruction for all combinations of relaxation correction parameters $\tau$ and Deconvolution Stage parameters $\nu_0$, with $\gamma = \gamma^*$ fixed for all reconstruction.}
		\label{subfig:exp1:taunu:ssim}
	\end{subfigure}
	\caption{Overview of the average PSNR values over the simulated phantoms.}
	\label{fig:exp1}
\end{figure}
From these tables we observe that the highest PSNR score obtained for the whole reconstruction is attained when the 
correction parameter $\tau$ coincides with the ground truth delay $\tau_{\mathrm{GT}} = 5\cdot 10^{-6}\ \si{\second}$.

\begin{figure}[t]
	\includegraphics[width=\linewidth]{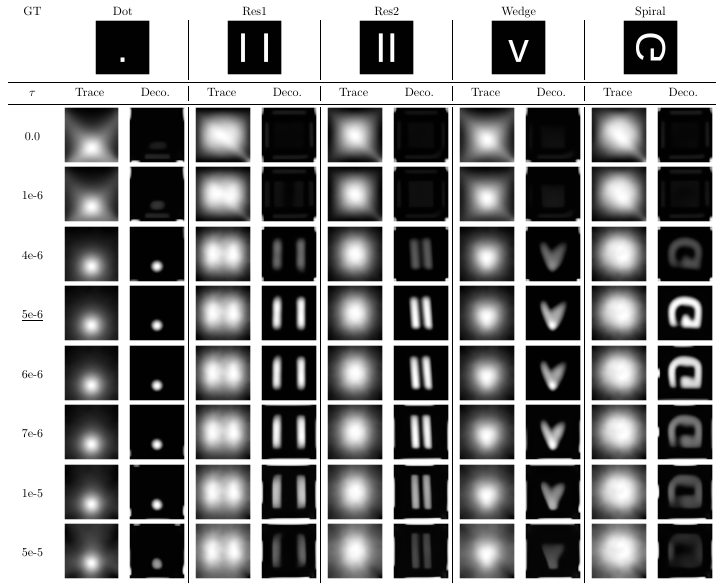}
	\caption{Reconstruction results after the MPI Core Stage (Trace) and the Deconvolution Stage (Deco.) of the 5 simulated phantoms 
		for a variety of relaxation adaption parameters $\tau$ (rows). The ground truth delay $\tau_{\mathrm{GT}} = 5\cdot 10^{-6}$ 
		is underlined. We have diplayed the traces obtained after the MPI Core Stage with the parameter $\gamma^*$. 
		The final reconstructions are obtained using the deconvolution parameter $\nu_0^* = 10^{-7}$ using all entries of 
		the MPI Core Response. We observe that the reconstructions obtained with $\tau = \tau_{\mathrm{GT}} = 5\cdot 10^{-6}$ 
		are among the closest to the ground truths and exhibit the smallest amount of artifacts.}
	\label{fig:simulations}
\end{figure}

For a qualitative underpinning of this proof of concept, we display the reconstructions obtained in 
figure~\ref{fig:simulations}. We observe that the reconstructions look closer to the ground truths when the 
relaxation adaption $\tau$ corresponds to the ground truth delay of $\tau_{\mathrm{GT}} = 5\cdot 10^{-6}$. 
Moreover, we observe that when the relaxation adaption parameter $\tau$ is too small, the reconstructed traces 
are strongly blurred and the deconvolution does not retrieve meaningful information about the underlying distribution. 
On the other hand, when the adaption parameter $\tau$ is too large, one observes the increasing number of boundary 
effects on the final reconstructions.

\begin{table}[t]
	\centering
	\setlength{\tabcolsep}{2pt}
	\renewcommand{\arraystretch}{1.15}
	
	\begin{tabular}{|l|
			c|c|c|c|c|c|
			c|c|c|c|
			c|c|c|c|c|c|}
		\hline
		
		& \multicolumn{6}{c|}{\textbf{Preprocessing}}
		& \multicolumn{4}{c|}{\textbf{Core Stage}}
		& \multicolumn{6}{c|}{\textbf{Deconvolution Stage}}
		\\ \hline

		\textbf{Phantom}
		
		& \rotatebox{90}{AF-TF}
		& \rotatebox{90}{M-TF}
		& \rotatebox{90}{HPF}
		& \rotatebox{90}{$\mathrm{SNR}_x$}
		& \rotatebox{90}{$\mathrm{SNR}_y$}
		& \rotatebox{0}{$\tau$}
		
		& \rotatebox{90}{Reg.~Ord.}
		& \rotatebox{90}{CG maxit}
		& \rotatebox{90}{CG tol.}
		& \rotatebox{0}{$\gamma$}
		
		& \rotatebox{90}{CG maxit}
		& \rotatebox{90}{CG tol.}
		& \rotatebox{0}{$n_{\mathrm{it}}$}
		& \rotatebox{0}{$\nu_0$}
		& \rotatebox{90}{\% cut}
		& \rotatebox{90}{\% padding}
		\\
		\hline
		
		\textbf{Dot}
		& $\cmark$ & $\xmark$ & 3 & 1 & 1 & $10^{-6}$
		& 2 & 15k & $10^{-6}$ & $10^{-7}$
		& 10k & $10^{-12}$ & 10 & $10^{-5}$ & 5 & 5
		\\ \hline
		
		\textbf{Icecream}
		& $\cmark$ & $\xmark$ &  2 & 1 & 1 & $10^{-6}$
		& 2 & 15k & $10^{-6}$ & $10^{-7}$
		& 10k & $10^{-12}$ & 10 & $10^{-5}$ & 5 & 5
		\\ \hline
		
		\textbf{Spiral}
		& $\cmark$ & $\xmark$ & 2 & 3 & 3 & $2 \cdot 10^{-6}$
		& 2 & 15k & $10^{-6}$ & $10^{-7}$
		& 10k & $10^{-12}$ & 10 & $10^{-5}$ & 5 & 5
		\\ \hline
		
	\end{tabular}%
	
	\caption{Overview of the parameters used in the reconstructions of the real MPI 2D data. 
		All reconstructions have been performed on a $51\times 45$ grid discretization 
		of the FOV using the data collected in the DF-FOV.}
	\label{tab:params}
\end{table}

\begin{figure}[t]
	\centering
	\begin{subfigure}[t]{0.49\textwidth}
		\includegraphics[width=\linewidth]{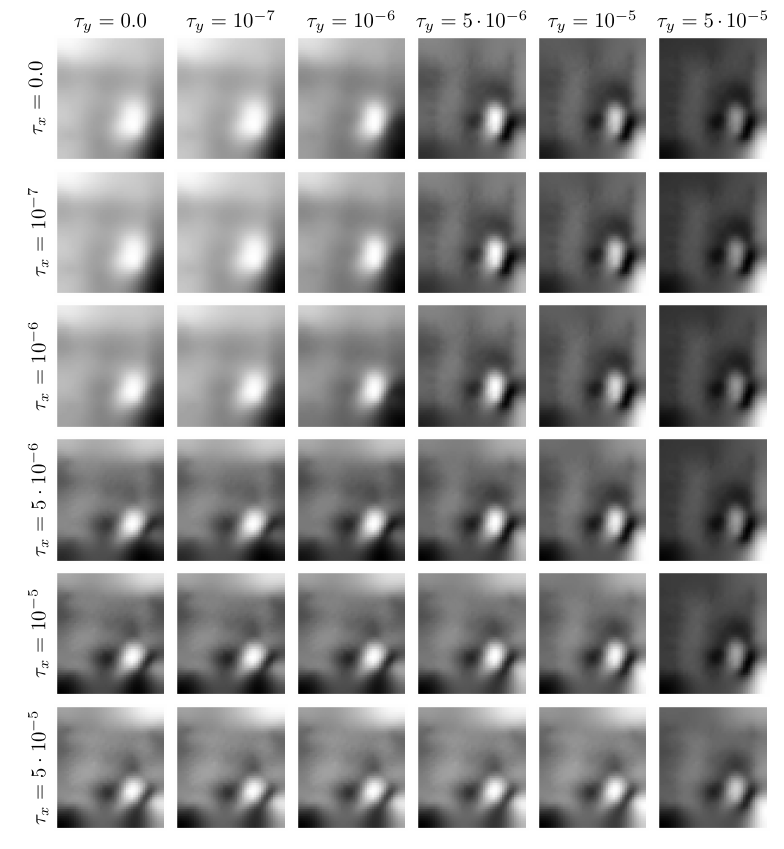}
		\caption{Traces with different $\tau_x$ and $\tau_y$.}
		\label{subfig:dot:traces}
	\end{subfigure}
	\hfil
	\begin{subfigure}[t]{0.49\textwidth}
		\includegraphics[width=\linewidth]{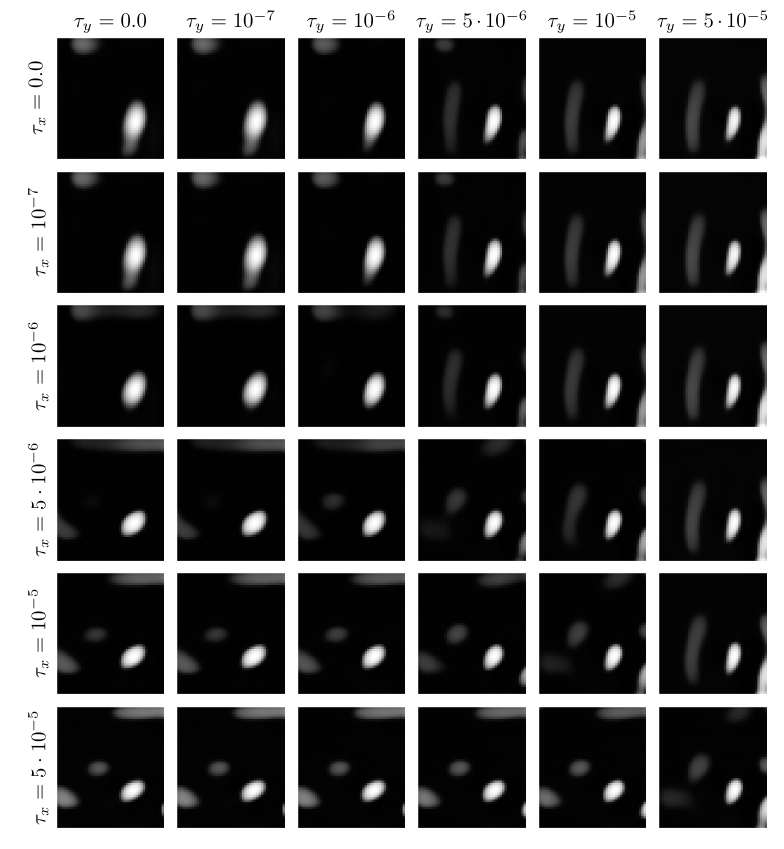}
		\caption{Reconstructions with different $\tau_x$ and $\tau_y$.}
		\label{subfig:dot:recos}
	\end{subfigure}
	\caption{Reconstruction results on the dot phantom in figure~\ref{subfig:dot} for a variety of 
		relaxation adaption parameters $\tau_x$ and $\tau_y$. All other parameters are fixed. We observe in 
		figure~\ref{subfig:dot:traces}, that when $\tau_x$, $\tau_y < 10^{-6}$, the trace and the subsequent deconvolution 
		are a blurry version of the dot phantom. Increasing the size of $\tau_x$, $\tau_y$ helps obtaining a trace and a reconstruction 
		in which the dot-phantom is less blurred. Observing the reconstructions in figure~\ref{subfig:dot:recos} row-wise and 
		column-wise suggests that increased correction in one of the channels helps obtaining less blur in that direction. 
		On the other hand, we observe that, when the relaxation parameters $\tau_x$, $\tau_y\geq 5\cdot 10^{-6}$, 
		reconstruction artifacts appear more strongly, the bigger the relaxation parameter.}
	\label{fig:dot:various:tau}
\end{figure}

\subsection{Effect of relaxation adaption on real data}
In this section reconstrcutions are performed on real data from the dataset~\cite{knopp2024equilibriumdata}.
Concerning the choice of a discretization grid, the FOV in the dataset~\cite{knopp2024equilibriumdata} 
is of size $34\ \si{\milli\meter}\times 30\si{\milli\meter}$ 
and the usual discretization is a $17\times 15$ grid~\cite{maass2024equilibriumanysotropy}. 
In the experiments of this section we consider a 3-fold refinement of the 
reconstruction grid $N_x\times N_y = 51\times 43$.
We note further that the FOV is larger than the drive-field field of view (DF-FOV), 
which is the smallest axis parallel box containing the scanning curve,
where the data is collected.

\subsubsection{Effect of channel-wise relaxation adaption on a real phantom.}
In this second experiment we consider the dot phantom contained in figure~\ref{subfig:dot}
and study the effect of relaxation adaption. 
To this end we perform reconstruction using the relaxation adaption step 
for a set of values $\tau \in \lbrace 0, 10^{-7},10^{-6}, 10^{-5},5\cdot 10^{-5}\rbrace$. 
All other parameters of the method are fixed according to table~\ref{tab:params}.
The relaxation adaption has been performed for each channel separately, so that we have corrected 
the $x$-channel with a relaxation parameter $\tau_x$ and the $y$-channel with a parameter $\tau_y$. 
In this way, the effect of the relaxation adaption on both the trace and the final deconvolution is more visible. 
An overview of the reconstructions is displayed in figure~\ref{fig:dot:various:tau}.

We observe from the traces in figure~\ref{subfig:dot:traces} that the effect of a lack of adaption (upper left quarter) 
corresponds to a blurring of the trace. This phenomenon is close to the one observed in the simulated scenario 
in \ref{fig:simulations}. On the other hand, a relaxation adaption parameter $\tau$ which is too high, yields artifacts. 
From the visual results obtained in the final reconstruction in figure~\ref{subfig:dot:recos}, it appears that the relaxation 
parameter that yields reconstruction that are visually similar to the underlying phantom in figure~\ref{subfig:dot} 
is close to the values of $10^{-6}\ \si{\second}$. Finally, observing the results in figure~\ref{fig:dot:various:tau} 
row-wise (resp. column-wise), the relaxation adaption step performed channel wise, seem to have a corresponding deblurring 
effect on the channel's direction. This phenomenon can be seen by inspecting any of the columns in~\ref{subfig:dot:recos}. 
For each column (fixed $\tau_y$), increase of the parameter $\tau_x$ corresponds to a compression of the reconstruction 
along the vertical direction (corresponding to the $x$-channel). Conversely, for any fixed $\tau_x$, inspection of the 
related row shows compression of the reconstruction in the horizontal direction (corresponding to the $y$-channel). 
Coherently, the diagonal entries of figure~\ref{fig:dot:various:tau} show similar compression of the reconstructed 
phantom along both directions. This indicates that the choice of a unique relaxation adaption parameter $\tau = \tau_x = \tau_y$ 
for both channels is sensible.

\subsubsection{Effect of the relaxation adaption on further phantoms}

In our last experiment, we show reconstructions on three phantoms in figure~\ref{fig:real:phantoms} and in particular, 
we show that the relaxation adaption step can help in reconstructing features of the underlying particle distributions 
that are not visible by considering simply the Langevin model. 
The chosen relaxation parameters as well as the preprocessing and reconstruction 
parameters are displayed in table~\ref{tab:params}. These parameters have been selected by visual inspection of 
the reconstructions obtained. The final reconstructions are displayed in figure~\ref{fig:recos}. We observe that the 
shape and the main features of the phantoms are reconstructed. 

\begin{figure}[!th]
	\centering
	\begin{subfigure}[t]{0.25\textwidth}
		\includegraphics[width=\linewidth]{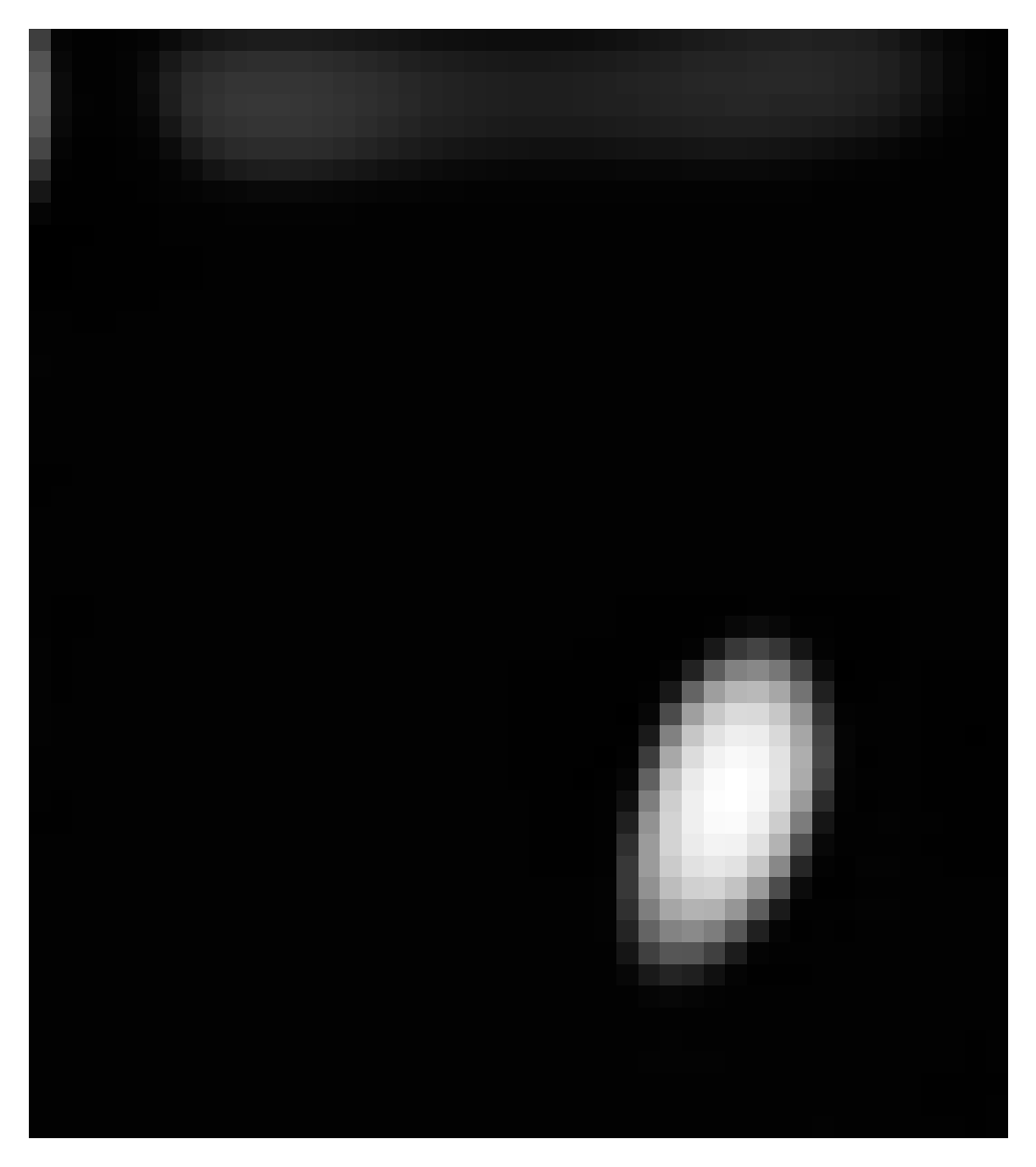}
		\caption{Dot, $\tau = 10^{-6}$}
		\label{subfig:reco:dot}
	\end{subfigure}
	\begin{subfigure}[t]{0.25\textwidth}
		\includegraphics[width=\linewidth]{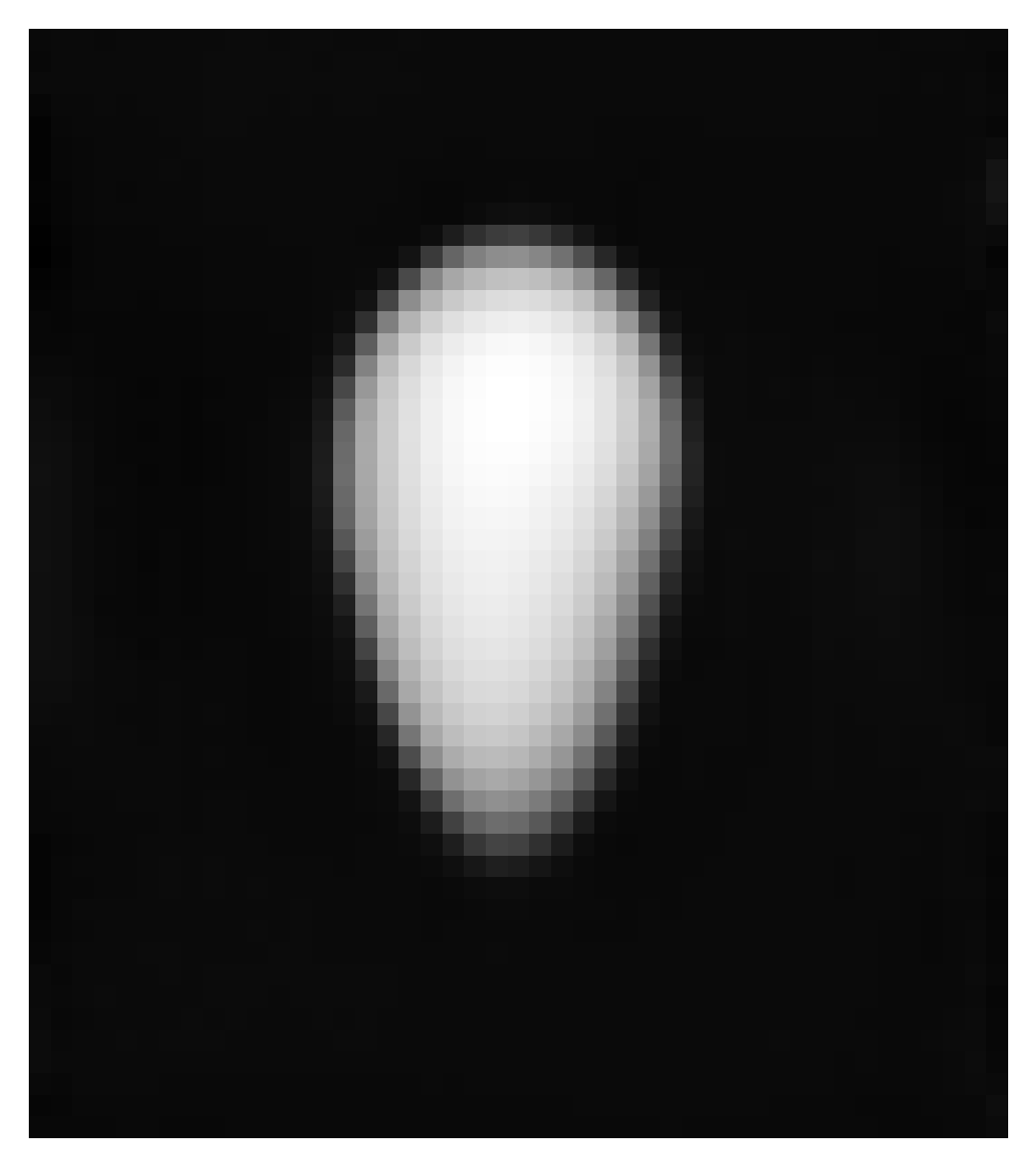}
		\caption{IceCream, $\tau = 10^{-6}$}
		\label{subfig:reco:icecream}
	\end{subfigure}
	\begin{subfigure}[t]{0.25\textwidth}
		\includegraphics[width=\linewidth]{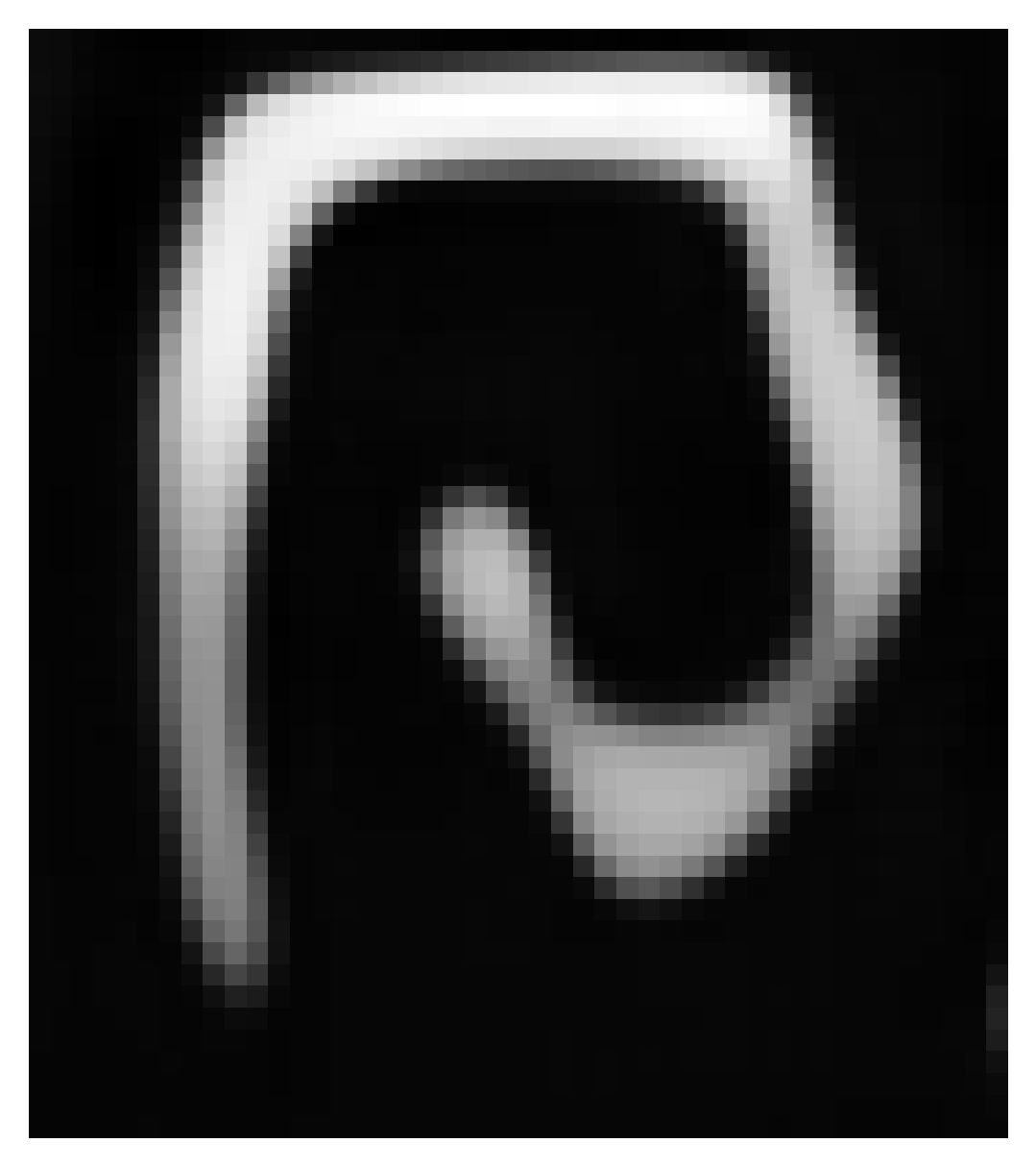}
		\caption{Snail, $\tau = 2\cdot 10^{-6}$}
		\label{subfig:reco:snail}
	\end{subfigure}
	\caption{Reconstruction obtained with the relaxation adaption step of the three phantoms in~\ref{fig:real:phantoms} from the real data in the dataset~\cite{knopp2024equilibriumdata}. Comparing the reconstructions with the pictures in figure~\ref{fig:real:phantoms}, we observe that the shape and the main features of the phantoms are reconstructed.}
	\label{fig:recos}
\end{figure}

To demonstrate that the reconstruction of the shapes and the features is an effect of the 
model and relaxation adaption step introduced in this paper, we show in figure~\ref{fig:snail:detail} in more detail 
the effect of the relaxation adaption step on the components of the MPI core response 
$A_{1,1}$, $A_{1,2}$, $A_{2,1}$ and $A_{2,2}$, on the trace $A_{1,1}+A_{2,2}$ as well as 
on the final reconstruction for a variety of choices of the parameter $\tau$.
\begin{figure}[!th]
	\centering
	\includegraphics[width=\linewidth]{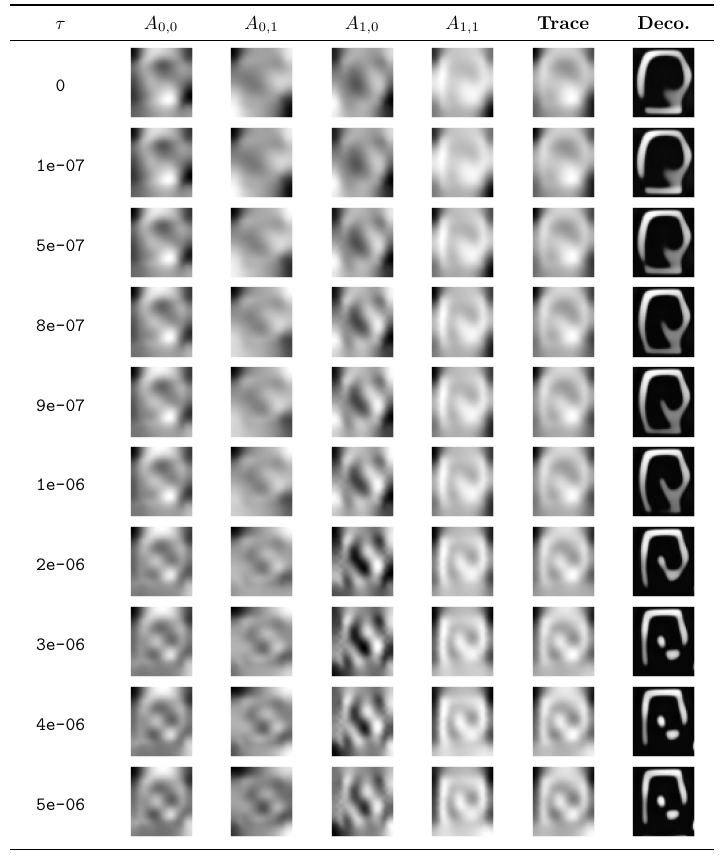}
	\caption{Detailed overview of the outputs of the MPI Core Stage and the Deconvolution Stage for a set of relaxation parameters $\tau$. We observe that, when there is no adaption ($\tau = 0$) or if $\tau$ is too small (below $5\cdot 10^{-7}$), the shape of the phantom is not clearly recognizable. An increased parameter, helps recover the shape, although with various artifacts. When $\tau = 2\cdot 10^{-6}$, the spiral shape of the phantom is clearly visible.  }
	\label{fig:snail:detail}
\end{figure}
From figure~\ref{fig:snail:detail} we observe that when the Langevin model is considered, i.e., no relaxation ($\tau=0$), 
the snail-like shape of the phantom in figure~\ref{subfig:snail} is not distinguishable. The snail phantom remains 
indistinguishable even if relaxation correction is applied but with a parameter $\tau$ that is too small 
(e.g. smaller than $5\cdot 10^{-7}$). However, when the relaxation parameter $\tau$ is big enough, the snail-like shape 
of the phantom becomes more visible, although with artifacts. When the relaxation parameter $\tau$ reaches $10^{-6}$, 
the shape of the phantom is clearly visible. An interesting observation is that if $\tau$ increases slightly beyond $10^{-6}$, 
the reconstructed phantom is not a unique, connected piece, but shows breaking points between the junctures of the tubes 
constituting the phantom. This is however interesting, as the phantom in figure~\ref{subfig:snail} is not a unique, bent and 
connected tube containing nanoparticles, but the junction of 5 independent tubes, disposed in a spiral pattern. 
From a closer inspection of the phantom in figure~\ref{subfig:snail}, there are no particles in the junctions between 
the tubular elements of the phantom.

\section{Conclusion}
In this work, we have considered a multi-dimensional Debye model to account 
for relaxation effects of the magnetic particles. 
We have obtained reconstruction formulae for this model
and shown that the Debye model-based signal is related to Langevin model-based signal
via a linear time-invariant system with exponential memory.
We have shown that this relation, upon time discretization, 
implies a first-order linear recurrence to yield the corresponding Langevin 
model-based signal data from the measured signal data.
We employed the latter to devise a relaxation-adaption step in order to provide a reconstruction algorithm. 
We have seen that the relaxation-adapted signal data is represented by the standard MPI core operator 
which allowed us to adapt our previous reconstruction scheme MoBiT-2S. 
Concerning computational costs, we have observed that, 
compared with the corresponding Langevin model based scheme,
the relaxation-adaption step in proposed reconstruction scheme adds 
only low computational cost linear in the number of time-samples $L$ collected during a scan. 
We have provided numerical results for the proposed algorithmic approach for both simulated and real data.
In particular, we were able to provide model-based reconstructions of fully 
multi-dimensional FFP data in a fully model-based setup. 
Neither a specific model transfer function as in the Langevin model approaches of 
\cite{droigk2022multidimcheb,gapyak2025trajectoryindep} nor hybrid corrections were employed. 
Instead, we were able to incorporate modeled relaxation effects of the particles via the Debye model 
to obtain reasonable reconstructions from real data. 

Topics of future research include the extension of the proposed approach to further MPI setups 
including FFL setups as well as the extension of the mathematical analysis of the proposed method. 
From an algorithmic side, imposing further regularization beyond discretization is a topic to be 
addressed when dealing with more general setups. 

\section*{Acknowledgment}

This work was supported by the Hessian Ministry of Higher Education, Research, Science and the Arts within the Framework of the ``Programm zum Aufbau eines akademischen Mittelbaus an hessischen Hochschulen" and by the German Science Fonds DFG under grant INST 168/4-1.

\bibliographystyle{ieeetr}
\bibliography{literature}

\end{document}